\documentclass[11pt]{article}

\usepackage[draft]{ed}
\usepackage{amsfonts}
\usepackage{amsmath}
\usepackage[linesnumbered,lined,boxed,commentsnumbered]{algorithm2e}
\usepackage{amssymb}
\usepackage{epsfig}
\usepackage{bm}
\usepackage{enumerate}
\numberwithin{equation}{section} 

\usepackage{color}
\usepackage[normalem]{ulem}

\newcommand{\PEC}{\mathrm{PEC}}
\newcommand{\TR}{\mathrm{T}}

\title{Domain Decomposition Methods based on quasi-optimal transmission operators for the solution of Helmholtz transmission problems}
\author{ Yassine Boubendir\thanks{Department of Mathematical Sciences, NJIT, e-mail:boubendi@njit.edu}, Carlos Jerez-Hanckes\thanks{Institute for Mathematical and Computational Engineering, School of Engineering, Pontificia Universidad Cat\'olica de Chile, Av. Vicuna Mackenna 4860, Macul, Santiago, Chile, e-mail: cjerez@ing.puc.cl}, Carlos P\'erez-Arancibia\thanks{Department of Mathematics, Massachusetts Institute of Technology, e-mail: cperezar@mit.edu}, Catalin Turc\thanks{Department of Mathematical Sciences, NJIT, e-mail:catalin.c.turc@njit.edu} }

\newtheorem{theorem}{Theorem}[section]
\newtheorem{lemma}[theorem]{Lemma}

\newenvironment{proof}{\hspace{0.5cm} {\bf Proof.}}
{$\quad {}_\blacksquare$\vspace{0.3cm}}

\date{}
\newcommand{\triple}[1]{{\left\vert\kern-0.25ex\left\vert\kern-0.25ex\left\vert #1 
    \right\vert\kern-0.25ex\right\vert\kern-0.25ex\right\vert}}

\begin{document}
\maketitle
\begin{abstract}
  We present non-overlapping Domain Decomposition Methods (DDM) based on quasi-optimal transmission operators for the solution of Helmholtz transmission problems with piece-wise constant material properties. The quasi-optimal transmission boundary conditions incorporate readily available approximations of Dirichlet-to-Neumann operators. These approximations consist of either complexified hypersingular boundary integral operators for the Helmholtz equation or square root Fourier multipliers with complex wavenumbers. We show that under certain regularity assumptions on the closed interface of material discontinuity, the DDM with quasi-optimal transmission conditions are well-posed. We present a DDM framework based on Robin-to-Robin (RtR) operators that can be computed robustly via boundary integral formulations. More importantly, the use of quasi-optimal transmission operators results in DDM that converge in small numbers of iterations even in the challenging high-contrast, high-frequency regime of Helmholtz transmission problems. Furthermore, the DDM presented in this text require only minor modifications to handle the case of transmission problems in partially coated domains, while still maintaining excellent convergence properties. We also investigate the dependence of the DDM iterative performance on the number of subdomains. 
   
\end{abstract}

\section{Introduction}
\label{intro}

The phenomenon of electromagnetic wave scattering by bounded penetrable objects with constant electric permittivities is relevant for numerous applications in antenna design, diffraction gratings, to name but a few. Numerical methods based on Boundary Integral Equations (BIE) are well-suited for simulation of these types of applications owing to the dimension reduction, explicit enforcement of radiation conditions, and lack of numerical dispersion that they enjoy. There is a wide array of well-conditioned BIE of the second kind for the solution of Helmholtz transmission problems, see references~\cite{turc2,boubendir2015regularized} for in-depth discussions. These formulations rely on regularization techniques that incorporate readily available approximations of Dirichlet-to-Neumann (DtN) operators. Nevertheless, these formulations require relative large numbers of Krylov subspace iterations  for convergence in the high-contrast, high-frequency regime. This situation can be attributed to the lack of easily computable approximations of DtN operators for bounded domains~\cite{boubendir2015regularized}. Furthermore, the regularization strategy becomes cumbersome in the case of more complicated material properties, e.g., perfectly conducting coatings, multiple junctions. A different alternative is to resort to matrix compression techniques to produce direct solvers for the solution of Helmholtz transmission problems in the computationally challenging high-contrast, high-frequency regime~\cite{greengard1,greengard2}. However, the direct solvers proposed in~\cite{greengard1,greengard2} require large memory consumption and are difficult to parallelize. 

Domain Decomposition Methods (DDM) are natural candidates for the solution of scattering problems involving composite scatterers. DDM are divide-and-conquer strategies whereby the computational domain is divided into smaller subdomains for which solutions are matched via transmission conditions on subdomain interfaces. The convergence of DDM for time-harmonic wave scattering applications depends a great deal on the choice of the transmission conditions that allow the exchange of information between adjacent subdomains. These interface transmission conditions should ideally allow information to flow out of a subdomain with as little as possible information being reflected back into the subdomain. Thus, the interface transmission conditions fall into the category of Absorbing Boundary Conditions (ABC). From this perspective, the ideal choice of transmission conditions on an interface between two subdomains is such that the impedance/transmission operator is the restriction to the common interface of the DtN operator corresponding to the adjacent subdomain. Traditionally, the interface transmission conditions were chosen as the classical first order ABC outgoing Robin/impedance boundary conditions~\cite{Depres,Collino1}. The convergence of DDM with classical Robin interface boundary conditions is slow and is adversely affected by the number of subdomains. Fortunately, the convergence of DDM can be considerably improved through incorporation of ABC that constitute higher order approximations of DtN operators in the form of second order approximations with optimized tangential derivative coefficients~\cite{Gander1}, square root approximations~\cite{boubendirDDM}, or other types of non-local transmission conditions~\cite{Collino1,steinbach2011stable}. Alternatively, so-called Perfectly Matched Layers can be used at subdomain interfaces~\cite{stolk2013rapidly}.

We devote our attention to DDM for Helmholtz transmission problems that use transmission operators given by approximations of DtN operators in the form of either (a) hypersingular Helmholtz BIO, (b) square root Fourier multipliers, and (c) scalar multiples of identity operators. A fundamental requirement in DDM is that the subdomain Helmholtz problems with Robin/impedance boundary conditions that incorporate the aforementioned transmission operators are well posed. This entails the complexification of the wavenumbers in the definition of the three types of transmission operators considered above; this modification ensures that the transmission operators enjoy certain coercivity property that suffices for the well posedness of subdomain Helmholtz problems with Robin boundary conditions. In practice, these complex wavenumbers in the definition of the transmission operators are chosen to optimize the rate of convergence of the ensuing DDM in the case of simple geometries (e.g circles, infinite waveguides) amenable to analytical modal analysis~\cite{boubendirDDM}. A judicious choice of the complex wavenumbers in the definition of transmission operators gives rise to DDM whose rate of convergence is virtually independent of frequency; for this reason the transmission operators described above with appropriate complex wavenumbers are referred to as quasi-optimal transmission operators~\cite{boubendirDDM}. The DDM can be easily recast in operator form using certain subdomain Robin-to-Robin (RtR) operators that map outgoing Robin data to incoming Robin data defined in terms of transmission operators. In the case when the transmission operators correspond to the classical first order ABC, the ensuing RtR operators turn out to be unitary, a key ingredient in establishing the well posedness of DDM~\cite{Depres,Collino1}. However, the transmission operators of type (a), (b), and (c) considered above do not give rise to unitary RtR maps, and the wellposedness of the ensuing DDM is more complicated. Nevertheless, we investigate in more detail the  RtR operators by expressing them in terms of Boundary Integral Operators (BIO), and we are able to establish the well posedness of DDM for Helmholtz transmission problems with one closed interface of discontinuity under various assumptions on the regularity of that interface. 

The key computational ingredient in the implementation of DDM is the computation of subdomain RtR maps. We provide several robust representations of RtR maps in terms of BIOs and their inverses. We present high-order discretizations of the RtR operators based on Nystr\"om discretizations of the BIO that enter in their representations. We provide ample evidence that the DDM based on the quasi-optimal transmission operators considered in this text give rise to small numbers of Krylov subspace iterations for Helmholtz transmission problems in the high-frequency, high-contrast regime. Furthermore, these numbers of iterations depend very mildly on the frequency, which is in stark contrast with solvers based on BIE~\cite{turc2}. Nevertheless, the computation of RtR maps becomes more problematic in the high-frequency regime, where matrices of large size need be inverted. Thus, it is customary to resort to divide the penetrable scatterer in a collection of non-overlapping subdomains and formulate a DDM that takes these further subdivisions into account. In particular, the quasi-optimal transmission operators need be restricted to subdomain interfaces that are open arcs. This is slighly delicate, given that those operators are global operators, i.e. their definition requires an integration boundary that is a closed curve. These restrictions can be effected by localizations via smooth cut-off functions supported on subdomain interfaces~\cite{jerez2017multitrace}. This strategy allows for an extension of DDM with quasi-optimal transmission operators in the presence of cross points between subdomains, i.e.~points where three or more domains with different material properties meet. In particular, we can use the cut-off methodology to formulate DDM with quasi-optimal transmission operator for the case of transmission problems in partially coated domains. Again, the DDM that incorporate quasi-optimal transmission operators perform well in terms of numbers of iterations in the case of penetrable scatterers that are partially coated.  

Although the use of quasi-optimal transmission operators, as those recounted above, accelerates a great deal the convergence of DDM, the number of iterations required for convergence still grows with the number of subdomains. This is not entirely surprising since the transmission operators are chosen to optimize the local exchange of information between adjacent subdomains, and affect to a lesser degree the global exchange of information between distant subdomains. Recent efforts have been directed to construct ``double sweep''-type preconditioners that address the latter issue~\cite{vion2014double,zepeda2016method}. The resulting preconditioned DDM scale favorably with frequency and number of subdomains, but appear to be somewhat less effective for wave propagation problems in composite media that exhibit sharp high-contrast interfaces. The incorporation of DDM preconditioners is currently underway.  

The structure of this paper is as follows. Section~\ref{MS10} describes the Helmholtz transmission problem. In Section~\ref{MS20} we review the BIOs associated with the Helmholtz equation and their mapping properties, as well as the classical boundary integral equations of the second kind for the solution of transmission problems. The DDM approach with three choices of transmission operators is then introduced and analyzed in Section~\ref{DDMg}. We present in Section~\ref{MS1} the transmission problem in partially coated obstacles as well as BIE and DDM formulations of such problems. Section~\ref{Nystrom} discusses high-order Nystr\"om discretizations of the Robin-to-Robin maps that are central to DDM, while a variety of numerical results are shown in Section~\ref{num}. Finally,  the conclusions of this work are presented in Section~\ref{conclu}.

\parskip 2pt plus2pt minus1pt

\section{Scalar transmission problems \label{MS10}}

We consider the problem of two dimensional scattering by penetrable homogeneous scatterers. Let $\Omega_1$ denote a bounded domain in $\mathbb{R}^2$ whose boundary $\Gamma:=\partial\Omega_1$ is a closed curve, and let $\Omega_0:=\mathbb{R}^2\setminus\Omega_1$. We seek to find fields $u_0$ and $u_1$ that are solutions of the following scalar Helmholtz transmission problem:
\begin{equation}\label{system_t}
\begin{array}{rclll}
 \Delta u_j +k_j^2 u_j &=&0& {\rm in}& \Omega_j,\quad j=0,1\smallskip\\
  u_0+u^{inc}&=&u_1&{\rm on}& \Gamma,\smallskip\\
  \alpha_0(\partial_{n_0}u_0+\partial_{n_0}u^{inc})&=&-\alpha_1\partial_{n_1}u_1&{\rm on}& \Gamma,\smallskip\\
  \multicolumn{4}{c}{\displaystyle \lim_{r\to\infty}r^{1/2}(\partial u_0/\partial r-ik_0u_0)=0.} \end{array}\end{equation}
We assume that the wavenumbers $k_j$ and the quantities $\alpha_j$ in the subdomains $\Omega_j$ are positive real numbers. The unit normal to the boundary $\partial\Omega_j$ is here denoted by $n_j$ and is assumed to point to the exterior of the subdomain $\Omega_j$. The incident field $u^{inc}$, on the other hand, is assumed to satisfy the homogeneous Helmholtz equation with wavenumber $k_0$ in the unbounded domain $\Omega_0$. Finally, we assume that the parameters $\alpha_j$ are positive so that the transmission problem~\eqref{system_t} is well posed under the assumption that $\Gamma$ is given locally by the graph of a Lipschitz function. The well posedness remains valid in the case when $\Gamma$ is more regular.

In what follows, we review two main formulations of the transmission problem~\eqref{system_t}. One formulation relies on BIE, while the other is is a DDM.

\section{Boundary integral equation formulations\label{MS20}}

There is a wide variety of possibilities in which equations~\eqref{system_t} can be reformulated via robust BIE, see contribution~\cite{dominguez2016well} for an in-depth discussion. We will present here a BIE of the second kind. To this end, we make use of the four BIO associated with the Calder\'on calculus. Let $D\subset\mathbb{R}^2$ be a bounded domain whose boundary $\partial D=\Gamma$ is a closed curve. In what follows we will focus on two cases: (1) $\Gamma$ is a $C^2$ curve (or smoother), and (2) $\Gamma$ is given locally by the graph of a Lipschitz function. Given a wavenumber~$k>0$, and a density $\varphi:\Gamma\to\mathbb{C}$, we recall the definitions of the single layer potential~
$$[\emph{SL}_{\Gamma,k}(\varphi)](\mathbf{z}):=\int_\Gamma G_k(\mathbf{z}-\mathbf{y})\varphi(\mathbf{y})ds(\mathbf{y}),\quad \mathbf{z}\in\mathbb{R}^2\setminus\Gamma,$$
and the double layer potential 
$$[\emph{DL}_{\Gamma,k}(\varphi)](\mathbf{z}):=\int_\Gamma \frac{\partial G_k(\mathbf{z}-\mathbf{y})}{\partial\mathbf{n}(\mathbf{y})}\varphi(\mathbf{y})ds(\mathbf{y}),\quad \mathbf{z}\in\mathbb{R}^2\setminus\Gamma,$$
where $G_k(\mathbf{x})=\frac{i}{4}H_0^{(1)}(k|\mathbf{x}|)$ denotes the free-space two-dimensional Green's function of the Helmholtz equation with wavenumber $k$, and $\mathbf{n}$ denotes the unit normal pointing outside the domain $D$. Applying exterior (resp. interior) Dirichlet and Neumann traces on $\Gamma$, which are denoted by $\gamma_\Gamma^{D,\rm{ext}}$ and $\gamma_\Gamma^{N,\rm{ext}}$ (resp. $\gamma_\Gamma^{D,\rm{int}}$ and $\gamma_\Gamma^{N,\rm{int}}$), respectively, to the single and double layer potentials, we define the four Helmholtz BIO: single layer ($S_{\Gamma,k}$),  double layer ($K_{\Gamma,k}$), adjoint double layer ($K_{\Gamma,k}^\top$) and hypersingular  ($N_{\Gamma,k}$) operators, which satisfy
\begin{equation}\label{traces}\begin{array}{ccc}
\displaystyle\gamma_\Gamma^{D,\rm{ext}} \emph{SL}_{\Gamma,k}(\varphi)=\gamma_\Gamma^{D,\rm{int}} SL_{\Gamma,k}(\varphi)=S_{\Gamma,k}\varphi,&& \displaystyle\gamma_\Gamma^{N,\rm{ext}} \emph{DL}_{\Gamma,k}(\varphi)=\gamma_\Gamma^{N,\rm{int}} DL_k(\varphi)=N_{\Gamma,k}\varphi,\medskip\\
\displaystyle\gamma_\Gamma^{N,\rm{ext}} SL_{\Gamma,k}(\varphi)=-\frac{\varphi}{2}+K_{\Gamma,k}^\top \varphi,&&
\displaystyle\gamma_\Gamma^{D,\rm{ext}} DL_{\Gamma,k}(\varphi)=\frac{\varphi}{2}+K_{\Gamma,k}\varphi,\medskip\\
\displaystyle\gamma_\Gamma^{N,\rm{int}} SL_{\Gamma,k}(\varphi)=\frac{\varphi}{2}+K_{\Gamma,k}^\top \varphi,&&
\displaystyle\gamma_\Gamma^{D,\rm{int}} DL_{\Gamma,k}(\varphi)=-\frac{\varphi}{2}+K_{\Gamma,k}\varphi.
\end{array}\end{equation}
Next, we replace  the subindex $k$ in the definition of the layer potentials and BIO by the subindex $j$ of the wavenumber $k_j$ corresponding to the $\Omega_j$ subdomain. We also denote the BIO associated with the Laplace equation---wavenumber equal to zero---by using the subindex $L$. 

For any $D\subset\mathbb{R}^2$ domain with bounded boundary $\Gamma$, we denote by $H^s(D)$ the classical Sobolev space of order $s$ on $D$~(\emph{cf.}~\cite[Ch. 3]{mclean:2000} or \cite[Ch. 2]{adams:2003}). If $\Gamma$ is of regularity $C^2$, the Sobolev spaces defined on the boundary $\Gamma$,  $H^s(\Gamma)$ are well defined for any $s\in[-3,3]$. If $\Gamma$ is a Lipschitz boundary, $H^s(\Gamma)$ is well defined for any $s\in[-1,1]$. We recall that for any $s>t$, $H^s(\Sigma)\subset H^t(\Sigma)$, $\Sigma\in\{D,\Gamma\}$ with compact support. Moreover, and
$\big(H^t(\Gamma)\big)'=H^{-t}(\Gamma)$ when the inner product of $H^0(\Gamma)=L^2(\Gamma)$ is used as duality product. Let $\Gamma_0\subset\Gamma$ such that $meas(\Gamma_0)>0$. For $0<s\leq 1/2$ we define by ${H}^s(\Gamma_0)$ be the space of distributions that are restrictions to $\Gamma_0$ of functions in $H^s(\Gamma)$. The space $\widetilde{H}^s(\Gamma_0)$ is defined as the closed subspace of $H^s(\Gamma_0)$
\[
\widetilde{H}^s(\Gamma_0):=\{u\in H^s(\Gamma_0):\widetilde{u}\in H^s(\Gamma)\},\ 0<s\leq 1/2,
\]
where
\[\widetilde{u}:=\begin{cases}
 u, & {\rm on}\  \Gamma, \\
 0,  & {\rm on}\ \Gamma\setminus\Gamma_0.
\end{cases}
\]
We define then $H^t(\Gamma_0)$ to be the dual of $\widetilde{H}^{-t}(\Gamma_0)$ for $-1/2\leq t<0$, and $\widetilde{H}^t(\Gamma_0)$ the dual of $H^{-t}(\Gamma_0)$ for $-1/2\leq t<0$.

We recount next several important results related to the mapping properties of the four BIO of the Calder\'{o}n calculus~\cite{dominguez2016well}. These mapping properties depend a great deal on the regularity of $\Gamma$. In the case when $\Gamma$ is a $C^2$ closed curve we have
\begin{theorem}\label{mappingS}
  Let  $D$ be a bounded domain in $\mathbb{R}^2$, with a boundary $\Gamma$ that is $C^2$. The following mappings
\begin{itemize}
\item $S_k:H^{s}(\Gamma)\to H^{s+1}(\Gamma)$
\item $K_k:H^{s}(\Gamma)\to H^{s+3}(\Gamma)$
\item $K^\top_k:H^{s}(\Gamma)\to H^{s+3}(\Gamma)$
\item $N_k:H^{s+1}(\Gamma)\to H^{s}(\Gamma)$
\end{itemize}
are continuous for $s\in[-3,0]$. Furthermore, if $k_1\ne k_2$ we have that  
\begin{itemize}
\item $S_{k_1}-S_{k_2}:H^{0}(\Gamma)\to H^{3}(\Gamma)$
\item $N_{k_1}-N_{k_2}:H^{0}(\Gamma)\to H^{1}(\Gamma)$.
\end{itemize}
are continuous. 
\end{theorem}

In the case of Lipschitz boundaries $\Gamma$, we will make use of the following mapping properties~\cite{dominguez2016well}: 
\begin{theorem}\label{mapping}
  Let  $D$ be a bounded domain in $\mathbb{R}^2$, with Lipschitz boundary $\Gamma$. The following mappings
\begin{itemize}
\item $S_k:H^{s}(\Gamma)\to H^{s+1}(\Gamma)$
\item $K_k:H^{s+1}(\Gamma)\to H^{s+1}(\Gamma)$
\item $K^\top_k:H^{s}(\Gamma)\to H^{s}(\Gamma)$
\item $N_k:H^{s+1}(\Gamma)\to H^{s}(\Gamma)$
\end{itemize}
are continuous for $s\in[-1,0]$. Furthermore, if $k_1\ne k_2$ we have that  
\begin{itemize}
\item $S_{k_1}-S_{k_2}:H^{-1}(\Gamma)\to H^{1}(\Gamma)$
\item $K_{k_1}-K_{k_2}:H^{0}(\Gamma)\to H^{1}(\Gamma)$
\item $K^\top_{k_1}-K^\top_{k_2}:H^{-1}(\Gamma)\to H^{0}(\Gamma)$
\item $N_{k_1}-N_{k_2}:H^{0}(\Gamma)\to H^{0}(\Gamma)$.
\end{itemize}
are continuous and compact. 
\end{theorem}

We also recount a result due to Escauriaza, Fabes and Verchota~\cite{EsFaVer:1992}. In this result, $K_L$, $K_L^\top$ are the double and adjoint double layer operator for Laplace equation (which obviously correspond to $k=0$).

\begin{theorem}\label{theo:inv}
For any Lipschitz curve $\Gamma$ and  $\lambda\not\in [-1/2,1/2)$, the mappings
\[
 \lambda I+K_L :H^s(\Gamma)\to H^s(\Gamma)
\]
are invertible for $s\in[-1,1]$. Furthermore, the mappings
\[
\frac{1}{2}I\pm K_L:H^s(\Gamma)\to H^s(\Gamma)
\]
are Fredholm of index zero for $s\in[-1,1]$.
\end{theorem}

BIE formulations of the transmission problem~\eqref{system_t} can be derived using layer potentials defined on $\Gamma$: the solutions $u_j,j=0,1,$ of the transmission problem are sought in the form:
\begin{equation}\label{eq:layer_t}
  u_j(\mathbf{x}):=SL_{\Gamma,j}\ v+(-1)^j\alpha_j^{-1} DL_{\Gamma,j}\ p,\quad \mathbf{x}\in\Omega_j,
\end{equation}
where $v$ and $p$ are densities defined on the $\Gamma$ and the double layer operators are defined with respect to exterior unit normals $\mathbf{n}$ corresponding to each domain $\Omega_j$. Applying Dirichlet and Neumann traces followed by transmission conditions, we arrive at the the following pair of integral equations:
\begin{equation}\label{eq:CFIESK}
\begin{array}{rcl}
\frac{\alpha_0^{-1}+\alpha_1^{-1}}{2}p -(\alpha_0^{-1}K_0+\alpha_1^{-1}K_1)p + (S_1 - S_0)v&=& \displaystyle u^{inc} \\
  \frac{\alpha_0+\alpha_1}{2}v+(N_0 - N_1)p+ (\alpha_0K_0^\top+\alpha_1K^\top_1)v&=& -\alpha_0\displaystyle \partial_{n_0} u^{inc}
\end{array}
\end{equation}
Note that the combination $N_0-N_1$ occurs, this is an integral operator with a weakly-singular kernel. In what follows we refer to the integral equations~\eqref{eq:CFIESK} by CFIESK. The well posedness of the CFIESK formulation in the space $(p,v)\in H^{0}(\Gamma)\times H^{0}(\Gamma)$ was established in~\cite{KressColton} in the case when $\Gamma$ is $C^2$. The well posedness of the CFIESK formulation in the space $(p,v)\in H^{1/2}(\Gamma)\times H^{-1/2}(\Gamma)$ was established in~\cite{ToWe:1993} in the case when $\Gamma$ is Lipschitz. There are several other possibilities to reformulate the transmission equation~\eqref{system_t} in terms of well-posed BIE~\cite{turc2,dominguez2016well}. We chose to focus on CFIESK formulations in this text, as these can be readily extended to more complex scenarios such as transmission problems in piece-wise constant composite domains that feature multiple junctions~\cite{claeys2015second,greengard1} or transmission problems in partially coated domains---see Section~\ref{MS1}.

 \section{Domain decomposition approach\label{DDMg}}

DDM are natural candidates for numerical solution of transmission problems~\eqref{system_t}. A non-overlapping domain decomposition approach for the solution of equations~\eqref{system_t} consists of solving subdomain problems in $\Omega_j,j=0,1$ with matching Robin transmission boundary conditions on the common subdomain interface $\Gamma$. Indeed, this procedure amounts to computing the subdomain solutions:

\begin{eqnarray}\label{DDM_t}
  \Delta u_j +k_j^2 u_j &=&0\qquad {\rm in}\quad \Omega_j,\\
  \alpha_j(\partial_{n_j}u_j+\delta_j^0\partial_{n_j}u^{inc})+Z_j(u_j+\delta_j^0 u^{inc})&=&-\alpha_\ell(\partial_{n_\ell}u_\ell+\delta_\ell^0 \partial_{n_\ell}u^{inc})+Z_j (u_\ell+\delta_\ell^0 u^{inc})\quad{\rm on}\quad \Gamma,\nonumber
\end{eqnarray}
where $\{j,\ell\}=\{0,1\}$ and $\delta_j^0$ stands for the Kronecker symbol, and $Z_j,Z_\ell$ are transmission operators with the following mapping property $Z_{j,\ell}:H^{1/2}(\Gamma)\to {H}^{-1/2}(\Gamma)$. The choice of the operators  $Z_j,Z_\ell$ should be such that the following PDEs are well posed
\begin{eqnarray}\label{domain_eqs_t}
   \Delta u_j+k_j^2u_j&=&0\ {\rm in}\ \Omega_j,\nonumber\\
   \alpha_j\partial_{n_j} u_j+Z_ju_j&=&\psi_j\ {\rm on}\ \Gamma,
\end{eqnarray}
for $j=0,1$, where we require in addition that $u_0$ be radiative at infinity. A sufficient condition for the well-posedness of these problems is given by
\begin{equation}\label{eq:well-pos_t}
  \pm\Im \int_{\Gamma}Z_1\varphi\ \overline{\varphi}ds>0\quad\mbox{and}\quad \Im \int_{\Gamma}Z_0\varphi\ \overline{\varphi}ds<0,\quad{\rm for\ all}\ \varphi\in H^{1/2}(\Gamma),
\end{equation}
under the assumption that $\alpha_j$ are positive numbers (cf.~\cite[Theorem 3.37]{KressColton}). In addition, $Z_0+Z_1:H^{1/2}(\Gamma)\to {H}^{-1/2}(\Gamma)$ must be a bijective operator in order to guarantee that the solution of the DDM system~\eqref{DDM_t} is also a solution of the original transmission problem~\eqref{system_t} (see Theorem~\ref{thm:wp_DDM}).  In order to describe the DDM method more concisely we introduce subdomain Robin-to-Robin (RtR) maps~\cite{Collino1}. For each subdomain $\Omega_j$, $j=0,1$, we define RtR maps $\mathcal{S}^j$, $j=0,1$, in the following manner:
\begin{equation}\label{RtRboxj_t}
   \mathcal{S}^0(\psi_0):=(\alpha_0\partial_{n_0} u_0-Z_1u_0)|_{\Gamma},\qquad \mathcal{S}^1(\psi_1):=(\alpha_1\partial_{n_1} u_1-Z_0u_1)|_{\Gamma}
 \end{equation}
 where $u_j$, $j=0,1$, are solutions of equations~\eqref{domain_eqs_t}. The DDM~\eqref{DDM_t} can be recast in terms of computing the global Robin data $f=[f_0\ f_1]^\top$ with
\[
f_{j}:=(\alpha_j\partial_{n_j}u_j+Z_j u_j)|_{\Gamma},\ j=0,1,
\]
as the solution of the following linear system that incorporates the subdomain RtR maps $\mathcal{S}^j,j=0,1$, previously defined
\begin{equation}\label{ddm_t}
 (I+\mathcal{S})f=g,\quad \mathcal{S}:=\begin{bmatrix}0&\mathcal{S}^1\\\mathcal{S}^0&0\end{bmatrix}
 \end{equation}
with right-hand side $g=[g_0\ g_1]^\top$ wherein
 \begin{eqnarray}\label{rhs_ddm_t}
   g_0&=& -(\alpha_0\partial_{n_0}u^{inc}+Z_0u^{inc})|_{\Gamma}\nonumber\\
   g_1&=&(-\alpha_0\partial_{n_0}u^{inc}+Z_1u^{inc})|_{\Gamma}.\nonumber
   \end{eqnarray}
 We note that due to its possibly large size, the DDM linear system~\eqref{ddm_t} is typically solved in practice via iterative methods. The behavior of iterative solvers of equations~\eqref{ddm_t} depends a great deal on the choice of transmission operators $Z_j$, $j=0,1$. Ideally, these transmission operators should be chosen so that information flows out of the subdomain and no information is reflected back into the subdomain. This can be achieved if the operator $Z_0$ is the Dirichlet-to-Neumann (DtN) operator corresponding to the Helmholtz equation~\eqref{domain_eqs_t} posed in the domain $\Omega_1$ and viceversa~\cite{Nataf,HJP13}. Since such DtN operators are not well defined for all wavenumbers $k_0$ and $k_1$, and expensive to calculate even when properly defined, easily computable approximations of DtN maps can be employed effectively to lead to faster convergence rates of GMRES solvers for DDM algorithms~\cite{boubendirDDM}. For instance, the transmission operators can be chosen in the following manner~\cite{turc2016well}:
 \begin{equation}\label{eq:calT_t}
 Z_0=-2\alpha_1N_{\Gamma,k_1+i\sigma_1},\qquad Z_1 = -2\alpha_0N_{\Gamma,k_0+i\sigma_0},\quad \sigma_j>0.
 \end{equation}
 Given that amongst Helmholtz BIOs, hypersingular operators are more expensive to compute, we proceed to replace the hypersingular operators in equation~\eqref{eq:calT_t} by principal symbol Fourier multiplier operators. The latter principal symbols are defined as
 \begin{equation}\label{eq:pN_t}
 p^N(\xi, k_0+i\sigma_0)=-\frac{1}{2}\sqrt{|\xi|^2-(k_0+i\sigma_0)^2}\quad {\rm and}\quad p^N(\xi, k_2+i\sigma_2)=-\frac{1}{2}\sqrt{|\xi|^2-(k_1+i\sigma_1)^2},
 \end{equation}
 where the square root branches are chosen such that the imaginary parts of the principal symbols are positive. The principal symbol Fourier multipliers are defined in the Fourier space $TM(\Gamma)$~\cite{AntoineX} as
 \begin{equation}
   [PS(N_{\Gamma,k_j+i\sigma_j})\varphi_1]\hat\ (\xi)=p^N(\xi, k_j+i\sigma_j)\hat{\varphi_1}(\xi)
 \end{equation}
 for a density $\varphi_1$ defined on $\partial\Omega_1$. We define accordingly
 \begin{equation}\label{eq:calPST_t}
 Z_0^{PS}=-2\alpha_1PS(N_{\Gamma,k_1+i\sigma_1}),\qquad Z_1^{PS} = -2\alpha_0PS(N_{\Gamma,k_0+i\sigma_0}),\quad \sigma_j>0,
 \end{equation}
 and use the operators in equation~\eqref{eq:calPST_t} as transmission operators in the DDM formulation. We refer to the ensuing DDM with transmission operators defined in~\eqref{eq:calPST_t} as Optimized DDM (DDMO). We note that given that both operators $Z_j$ and $Z_j^{PS}$ satisfy a G\aa rding inequality for $j=0,1$, it follows that $Z_0+Z_1$ as well as $Z_0^{PS}+Z_1^{PS}$ also satisfy  G\aa rding inequalities, and thus the latter operators are also invertible as operators from $H^{1/2}(\Gamma)$ to $H^{-1/2}(\Gamma)$. In addition, a high-frequency approximation as $k_j\to\infty$ of the square root expressions defined in equations~\eqref{eq:pN_t} results in yet another possible choice of transmission operators
 \begin{equation}\label{eq:damping}
   Z_0^a=-i\alpha_1(k_1+i\sigma_1)I\qquad Z_1^a=-i\alpha_0(k_0+i\sigma_0)I,
 \end{equation}
 where $I$ denotes the identity operator. The transmission operators defined in equation~\eqref{eq:damping} were originally introduced in a DDM setting in~\cite{boubendir2000domain}. We study in this paper the well-posedness of the DDM system~\eqref{ddm_t} with the aforementioned choices of transmission operators~\eqref{eq:calT_t},\phantom{a}\eqref{eq:calPST_t}, and~\eqref{eq:damping}. To the best of our knowledge, the first proof regarding the well-posedness of DDM  with Robin transmission for Helmholtz problems condition was provided in~\cite{Collino1} with $Z_j=i\eta,\ \eta<0$. In that case the RtR operators turn out to be unitary, a property that plays a crucial role in the well-posedness proof. In our case, neither of the choices presented above --i.e. equations~\eqref{eq:calT_t},\eqref{eq:calPST_t}, and~\eqref{eq:damping}-- leads to unitary RtR operators, and thus the proof of well-posedness of the DDM system~\eqref{ddm_t} should rely on different arguments. To this end, we look closer into the nature of the RtR operators by deriving exact representations of those in terms of boundary integral operators.

 \subsection{Calculations of RtR operators in terms of boundary integral operators~\label{rtr}}

The RtR operators $\mathcal S^0$ and $\mathcal S^1$ can be expressed in terms of solutions of the following Helmholtz problems
 \begin{eqnarray*}
   \Delta u_j+k_j^2u_j&=&0\ {\rm in}\ \Omega_j,\nonumber\\
   \partial_{n_j} u_j+\alpha_j^{-1}Z_ju_j&=&\varphi_j\ {\rm on}\ \Gamma,
 \end{eqnarray*}
for $j=0,1$, and with $u_0$ radiative at infinity, for which
\begin{equation*}
   \mathcal{S}^0(\varphi_0):=(\partial_{n_0} u_0-\alpha_0^{-1}Z_1u_0)|_{\Gamma}\quad\mbox{and}\quad \mathcal{S}^1(\varphi_1):=(\partial_{n_1} u_1-\alpha_1^{-1}Z_0u_1)|_{\Gamma}.
 \end{equation*}
 It turns out that the operators $\mathcal{S}^1$ can be computed robustly in a straightforward manner. Indeed, we start with Green's identity
\[
u_1=-DL_1(u_1|_\Gamma)+SL_1(\partial_{n_1}u_1|_\Gamma),\qquad{\rm in}\ \Omega_1
\]
to which we apply the Dirichlet trace on $\Gamma$ to derive another {\em direct} boundary integral equation
\begin{equation}\label{eq:int_D}
  \mathcal{B}_1u_1|_\Gamma=S_1\varphi_1,\ {\rm on}\ \Gamma\quad\mbox{where}\quad \mathcal{B}_1u_1|_\Gamma:=\left(\frac{1}{2}I+K_1+\alpha_1^{-1}S_1Z_1\right)u_1|_\Gamma.
\end{equation}
We establish the following result
\begin{theorem}\label{eq:inv_1}
 The operator $\mathcal{B}_1$ defined in equation~\eqref{eq:int_D} with $Z_1=-2\alpha_0 N_{k_0+i\sigma_0}$  is invertible with continuous inverse in the spaces $H^s(\Gamma)$ for all $s\in[-3,3]$ in the case when $\Gamma$ is $C^2$. In the case of Lipschitz $\Gamma$, the operator $\mathcal{B}_1$ defined in equation~\eqref{eq:int_D} is invertible with continuous inverse in the spaces $H^s(\Gamma)$ for all $s\in[-1,1]$.
\end{theorem}
\begin{proof}
  We will start by establishing the Fredholm property of $\mathcal{B}_1$ in the case of Lipschitz $\Gamma$ as the arguments are slightly more involved in this case. From Calder\'on identities we have that $\mathcal B^1$ can be expressed as 
 \begin{eqnarray*}
    \mathcal{B}_1&=&\frac{1}{2}I+K_1-2\frac{\alpha_0}{\alpha_1}S_1N_{k_0+i\sigma_0}\\
    &=&\frac{1}{2}I+K_L+\frac{\alpha_0}{2\alpha_1}I-2\frac{\alpha_0}{\alpha_1}K^2_L+\widetilde{B}_1,
  \end{eqnarray*}
where in what follows the BIOs with subscript $L$ denote the BIO corresponding to the Laplace equation, and
 \begin{eqnarray*}
    \widetilde{\mathcal{B}}_1&:=&(K_1-K_L)-2\frac{\alpha_0}{\alpha_1}S_1(N_{k_0+i\sigma_0}-N_L)+2\frac{\alpha_0}{\alpha_1}(S_1-S_L)N_L.
  \end{eqnarray*}
  
  Using the mapping properties recounted in Theorem~\ref{mapping} it follows immediately that the operator $\widetilde{\mathcal{B}}_1:L^2(\Gamma)\to H^1(\Gamma)$ and thus it is a compact operator in $L^2(\Gamma)$. On the other hand, we can establish the following identity
  \[
  \widetilde{\mathcal{B}}_2:=\frac{1}{2}I+K_L+\frac{\alpha_0}{2\alpha_1}I-2\frac{\alpha_0}{\alpha_1}K^2_L=-2\frac{\alpha_0}{\alpha_1}\left(\frac{1}{2}I+K_L\right)\left(-\frac{\alpha_0+\alpha_1}{2\alpha_0}+K_L\right)
  \]
  and thus the operator $\widetilde{\mathcal{B}}_2$ is the product of an operator that is Fredholm of index zero and an invertible operator (indeed, since $\frac{\alpha_0+\alpha_1}{2\alpha_0}>\frac{1}{2}$, we can apply the results in Theorem~\ref{theo:inv}), and hence $\widetilde{\mathcal{B}}_2$ is itself Fredholm of index zero in $L^2(\Gamma)$. Consequently, the operator $\mathcal{B}_1$ is a compact perturbation of a Fredholm operator of index zero in $L^2(\Gamma)$. In the case when $\Gamma$ is $C^2$, we use the decompositions above and we take advantage of the increased regularity of the double layer operators $K_L$ recounted in Theorem~\ref{mappingS} to establish
  \begin{equation}
    \mathcal{B}_1=\frac{\alpha_0+\alpha_1}{2\alpha_1}I+\widetilde{\mathcal{B}_3}
    \end{equation}
where the operator $\widetilde{\mathcal{B}}_3:L^2(\Gamma)\to H^1(\Gamma)$, and thus it is compact in $L^2(\Gamma)$. Hence, the operator $\mathcal{B}_1$ is a compact perturbation of a multiple of the identity operator in the space $L^2(\Gamma)$ in the case when $\Gamma$ is $C^2$. 

  The conclusion of the theorem follows once we establish the injectivity of the operator $\mathcal{B}_1$. The arguments are identical for both cases of boundary $\Gamma$ considered. Let $\psi\in\ Ker(\mathcal{B}_1)$ and let us define
  \[
  w:=DL_1\psi-2\frac{\alpha_0}{\alpha_1}SL_1[N_{k_0+i\sigma_0}]\psi,\qquad {\rm in}\ \mathbb{R}^2\setminus\Gamma.
  \]
  It follows that $\gamma_\Gamma^{D,ext}w=0$ and hence, from the uniqueness results for the exterior Dirichlet problem (cf.~\cite[Theorem 3.21]{KressColton}),  we obtain that  $w=0$ in $\Omega_0$. Using relations~\eqref{traces} we derive
  \[
  \gamma_\Gamma^{D,int}w=-\psi\qquad \gamma_\Gamma^{N,int}w=-2\frac{\alpha_0}{\alpha_1}N_{k_0+i\sigma_0}\psi.
  \]
  Using Green's identities we obtain
  \[
  \int_{\Omega_1}(|\nabla w|^2-k_1^2 w)dx = 2\frac{\alpha_0}{\alpha_1}\int_\Gamma (N_{k_0+i\sigma_0}\psi)\ \overline{\psi}\ ds.
  \]
  Using the fact that~\cite{turc2}
  \[
  \Im\int_\Gamma (N_{k_0+i\sigma_0}\psi)\ \overline{\psi}\ ds>0,\quad \psi\neq 0
  \]
  we obtain that $\psi=0$ which conclude the proof of the theorem in the space $L^2(\Gamma)=H^0(\Gamma)$. Clearly, the arguments of the proof can be repeated verbatim in the Sobolev spaces $H^s(\Gamma)$ for all $s\in[-1,0)$ in the case when $\Gamma$ is Lipschitz. The result in the remaining Sobolev spaces $H^s(\Gamma),\ s\in(0,1]$ follows then from duality arguments. Similar arguments hold in the case when $\Gamma$ is $C^2$.
\end{proof}

Once the invertibility of the operator $\mathcal{B}_1$ is established, we immediately obtain a representation of the corresponding RtR operator
\begin{equation}\label{eq:S1_B1}
  \mathcal{S}^1=I-\alpha_1^{-1}(Z_0+Z_1)\mathcal{B}_1^{-1}S_1.
  \end{equation}
The result established in Theorem~\ref{eq:inv_1} remains valid in the case of impedance operators $Z_1^a$, yet there are certain differences that we will comment on in the proof of Theorem~\ref{thm:wp_DDM_a}. In the case when $\Gamma$ is $C^2$, one can establish the compactness of the difference operator $N_{k_0+i\sigma_0}-PS(N_{k_0+i\sigma_0})$ in the space $H^1(\Gamma)$~\cite{turc2}, and the conclusion of Theorem~\ref{eq:inv_1} remains valid in the case of impedance operator $Z_1^{PS}$. Whether the aforementioned compactness property of the difference operator holds in the case of Lipschitz curves $\Gamma$ is an open question. We note that the arguments in the proof of Theorem~\ref{eq:inv_1} go through in the case of the exterior domain $\Omega_0$ provided that $k_0$ is not an eigenvalue of the Laplacean with Dirichlet boundary conditions in the domain $\Omega_1$. However, the well-posedness of the formulation in Theorem~\ref{eq:inv_1} cannot be establish for all positive wavenumbers $k_0$. This is not altogether surprising, as we have applied only Dirichlet traces to the Green's identities in order to derive formulations~\eqref{eq:int_D}. If we combine the application of Dirichlet and Neumann traces to Green's identities, the latter preconditioned on the left by suitable regularizing operators~\cite{turc2016well} we can derive a well-posed {\em direct} boundary integral equation of the second kind for the solution of both interior and exterior impedance boundary value Helmholtz problems. These formulations are expressed in the form
\begin{eqnarray}\label{eq:CFIER2}
\mathcal{A}_{j}(u_j|_\Gamma)&=&(S_j+S_{\kappa_j}-2S_{\kappa_j} K_j^\top)\varphi_j,\quad \kappa_j=k_j+i\sigma_j,\ \sigma_j>0,\nonumber\\
\mathcal{A}_{j}&:=&\frac{1}{2}I-2S_{\kappa_j} N_j+ \alpha_j^{-1}S_{\kappa_j}Z_j-2 \alpha_j^{-1}S_{\kappa_j} K_j^\top Z_j +K_j+\alpha_j^{-1}S_jZ_j. 
\end{eqnarray}
It is a straightforward matter~\cite{turc2016well} to show that in the case when $\Gamma$ is Lipschitz one can use the decomposition:
\begin{equation}\label{eq:A_comp}
  \mathcal{A}_j=\alpha_j^{-1}(\alpha_0+\alpha_1)I+\alpha_j^{-1}(\alpha_j-\alpha_{j+1})K_L-2\alpha_j^{-1}(\alpha_j+2\alpha_{j+1})K_L^2+4\alpha_j^{-1}\alpha_{j+1}K_L^3+\widetilde{\mathcal{A}_j}
\end{equation}
where the operators $\widetilde{\mathcal{A}_j}:L^2(\Gamma)\to H^1(\Gamma)$, and hence $\widetilde{\mathcal{A}_j}:L^2(\Gamma)\to L^2(\Gamma)$ are compact for $j=0,1$, and $j+1=j+1 (\!\!\!\mod 2)$. In the case when $\Gamma$ is $C^2$, the decomposition can be simplified in the form
\begin{equation}\label{eq:A_compS}
  \mathcal{A}_j=\alpha_j^{-1}(\alpha_0+\alpha_1)I+\widetilde{\mathcal{A}_j^{reg}},
\end{equation}
where the operators $\widetilde{\mathcal{A}_j^{reg}}:L^2(\Gamma)\to H^1(\Gamma)$, and thus are compact in $L^2(\Gamma)$ for $j=0,1$. In both instances the RtR operators $\mathcal{S}^j$ can be expressed as
\begin{equation}\label{eq:SjBI}
  \mathcal{S}^j=I-\alpha_j^{-1}(Z_0+Z_1)\mathcal{A}_j^{-1}(S_j+S_{\kappa_j}-2S_{\kappa_j} K_j^\top),\ j=0,1.
\end{equation}

 Another possibility to derive robust BIE formulations for the solution of impedance boundary value problems with impedance operators $Z_j$, $Z_j^{PS}$, and $Z_j^a$ was proposed in~\cite{steinbach2011stable}. This approach consists of applying both Dirichlet and Neumann traces to Green's identities in order to derive a system of boundary integral equations whose unknowns are the Cauchy data on the boundary $\Gamma$. Besides their simplicity, these formulations have the advantage of being well-posed for all three choices of impedance operators above and Lipschitz $\Gamma$. We start our presentation with the case of the bounded domain $\Omega_1$. Applying the interior Dirichlet and Neumann traces to Green's identity in the domain $\Omega_1$ we obtain
\begin{eqnarray*}
  \left(\frac{1}{2}I + K_1\right)u_1|_\Gamma -S_1\partial_{n_1}u_1|_\Gamma&=&0,\\
  -N_1u_1|_\Gamma+\left(-\frac{1}{2}I+K_1^\top\right)\partial_{n_1}u_1|_\Gamma&=&0.
\end{eqnarray*}
Adding to the second equation above the impedance boundary condition we derive the following system of BIE
\begin{equation}\label{eq:Steinbach}
  \begin{bmatrix}-\alpha_1^{-1}Z_1+N_1 & -\frac{1}{2}I-K_1^\top\\ -\frac{1}{2}I - K_1 & S_1 \end{bmatrix} \begin{bmatrix}u_1|_\Gamma\\\partial_{n_1}u_1|_\Gamma\end{bmatrix} =\begin{bmatrix}\varphi_1\\0 \end{bmatrix}.
\end{equation}
The well-posedness of the formulation~\eqref{eq:Steinbach} can be established by making use of the bilinear form 
\[
\langle (f,\varphi),(g,\psi) \rangle:= \int_\Gamma fg+\int_\Gamma \varphi\psi,\quad (f,\varphi)\in H^{1/2}(\Gamma)\times H^{-1/2}(\Gamma),\ (g,\psi)\in H^{-1/2}(\Gamma)\times H^{1/2}(\Gamma)
\] 
and following the same arguments presented in~\cite{steinbach2011stable}. We thus arrive at the following result whose proof can be obtained from a simple adaptation of the proof of Theorem 5.25 in~\cite{windisch2011boundary}:
\begin{theorem}\label{th:inv_3}
  The operator
  \[
  \mathcal{C}_1:=\begin{bmatrix}-\alpha_1^{-1}Z_1+N_1 & -\frac{1}{2}I-K_1^\top\\ -\frac{1}{2}I - K_1 & S_1 \end{bmatrix},\qquad \mathcal{C}_1: H^{1/2}(\Gamma)\times H^{-1/2}(\Gamma)\to H^{-1/2}(\Gamma)\times H^{1/2}(\Gamma)
  \]
  is invertible and its inverse is continous when $\Gamma$ is Lipschitz.
\end{theorem}

The equivalent formulation~\eqref{eq:Steinbach} cannot be shown to be well-posed in the case of the analogous impedance boundary value problem in the the exterior domain $\Omega_0$, unless $k_0$ is not an eigenvalue of the Laplacean with Dirichlet boundary conditions in $\Omega_1$. The remedy is to consider the following system of integral equations
\begin{equation}\label{eq:Steinbach0}
  \begin{bmatrix}-\alpha_0^{-1}Z_0+N_0 & -\frac{1}{2}I-K_0^\top\\ \alpha_0^{-1}S_{k_0+i\sigma_0}Z_0-\frac{1}{2}I - K_0 & S_0 + S_{k_0+i\sigma_0} \end{bmatrix} \begin{bmatrix}u_0|_\Gamma\\\partial_{n_0}u_0|_\Gamma\end{bmatrix} =\begin{bmatrix}\varphi_0\\S_{k_0+i\sigma_0}\varphi_0\end{bmatrix}
\end{equation}
whose derivation is absolutely similar to that of equations~\eqref{eq:Steinbach} except that we add to both sides of the second equation in~\eqref{eq:Steinbach} the identity
\[
\alpha_0^{-1}S_{k_0+i\sigma_0}Z_0u_0+S_{k_0+i\sigma_0}\partial_{n_0}u_0=S_{k_0+i\sigma_0}\varphi_0.
\]
In that case we have the following result whose proof follows from the same arguments as in the proof of Lemma 5.29 in~\cite{windisch2011boundary}:
\begin{theorem}\label{th:inv_4}
  The operator
  \[
  \mathcal{C}_0:=\begin{bmatrix}-\alpha_0^{-1}Z_0+N_0 & -\frac{1}{2}I-K_0^\top\\ \alpha_0^{-1}S_{k_0+i\sigma_0}Z_0-\frac{1}{2}I - K_0 & S_0 + S_{k_0+i\sigma_0} \end{bmatrix}, 
    \]
with the mapping property  $\mathcal{C}_0: H^{1/2}(\Gamma)\times H^{-1/2}(\Gamma)\to H^{-1/2}(\Gamma)\times H^{1/2}(\Gamma)$ is invertible with continuous inverse when $\Gamma$ is Lipschitz. 
\end{theorem}

Again, the result established in Theorem~\ref{th:inv_3} remains in the case of impedance operators $Z_1^a$ as well as $Z_1^{PS}$. Indeed, the key ingredient in establishing that result is the coercivity of the principal part of the operators $\mathcal{C}_1$, coercivity which is enjoyed by both operators $Z_1^a$ and $Z_1^{PS}$. On the other hand, the result in Theorem~\ref{th:inv_4} remains valid in the case of impedance operators $Z_0^a$ and a simple modification renders it valid in the case of impedance operators $Z_0^{PS}$ as well. This modification consists of replacing the single layer operators $S_{k_0+i\sigma_0}$ by Fourier multipliers whose principal symbols are the reciprocal of $p^N(\xi, k_0+i\sigma_0)$. Clearly, the results established in Theorem~\ref{th:inv_3} and Theorem~\ref{th:inv_4} remain valid in the case when $\Gamma$ is $C^2$.

Having discussed various strategies to derive robust BIE formulations of RtR operators $\mathcal{S}^j$, $j=0,1$, we next turn our attention to the well-posedness of the DDM formulation~\eqref{ddm_t}. 

\subsection{Well-posedness of the DDM formulation~\eqref{ddm_t}\label{DDMw}}

The well-posedness of the DDM formulation~\eqref{ddm_t} in the space $L^2(\Gamma)$ (and all $H^s(\Gamma), s\in[-1,1]$ in the case when $\Gamma$ is Lipschitz and all $H^s(\Gamma), s\in[-3,3]$ in the case when $\Gamma$ is $C^2$) hinges on the invertibility of the operator
\[
I-\mathcal{S}^0\mathcal{S}^1:L^2(\Gamma)\to L^2(\Gamma)
\]
via the formula
\begin{equation}\label{inv_matrix_explicit}
  (I+\mathcal{S})^{-1}=\begin{bmatrix}
I+\mathcal{S}^1(I-\mathcal{S}^0\mathcal{S}^1)^{-1}\mathcal{S}^0 & -\mathcal{S}^1(I-\mathcal{S}^0\mathcal{S}^1)^{-1}\\
-(I-\mathcal{S}^0\mathcal{S}^1)^{-1}\mathcal{S}^0 &
(I-\mathcal{S}^0\mathcal{S}^1)^{-1}
\end{bmatrix}.
\end{equation}
The invertibility of the operator $I-\mathcal{S}^0\mathcal{S}^1$, in turn, can be established via Fredholm arguments. In the case of more regular boundaries $\Gamma$, the situation is somewhat simpler, since 
\begin{lemma}\label{lemma1}
The RtR operators $\mathcal{S}^j:L^2(\Gamma)\to L^2(\Gamma)$ corresponding to the impedance operators $Z_j$ and $Z_j^{PS}$, $j=0,1$, are compact when the boundary $\Gamma$ is $C^2$.
\end{lemma}
\begin{proof}
We start from formula~\eqref{eq:SjBI} and we get
\begin{eqnarray*}
  \mathcal{S}^j&=&I-\alpha_j^{-1}(Z_0+Z_1)\mathcal{A}_j^{-1}(S_j+S_{\kappa_j}+2S_{\kappa_j}K_j^\top)=(Z_0+Z_1)\mathcal{A}_j^{-1}\mathcal{A}_j^1(Z_0+Z_1)^{-1},\\
  \mathcal{A}_j^1&:=&\mathcal{A}_j-\alpha_j^{-1}(S_j+S_{\kappa_j}+2S_{\kappa_j}K_j^\top)(Z_0+Z_1).
\end{eqnarray*}
A closer look into the operator $\mathcal{A}_j^1$ reveals via the decomposition~\eqref{eq:A_compS}
\begin{eqnarray*}
  \mathcal{A}_j^1&=&\alpha_j^{-1}(\alpha_0+\alpha_1)I+2\alpha_j^{-1}(S_j-S_L+S_{\kappa_j}-S_L)(\alpha_1N_{\kappa_1}+\alpha_0N_{\kappa_2})\nonumber\\
  &+&4\alpha_j^{-1}(\alpha_0+\alpha_1)S_LN_L+4\alpha_j^{-1}S_L(\alpha_1(N_{\kappa_1}-N_L)+\alpha_0(N_{\kappa_0}-N_L))\\
  &-&2\alpha_j^{-1}S_{\kappa_j}K_j^\top(Z_0+Z_1)+\widetilde{A_j^{reg}},
\end{eqnarray*}
which can be further simplified after using Calder\'on's identities
\begin{eqnarray*}
  \mathcal{A}_j^1&=&2\alpha_j^{-1}(S_j-S_L+S_{\kappa_j}-S_L)(\alpha_1N_{\kappa_1}+\alpha_0N_{\kappa_2})\\
  &+&4\alpha_j^{-1}(\alpha_0+\alpha_1)K_L^2+4\alpha_j^{-1}S_L(\alpha_1(N_{\kappa_1}-N_L)+\alpha_0(N_{\kappa_0}-N_L))\\
  &-&2\alpha_j^{-1}S_{\kappa_j}K_j^\top(Z_0+Z_1)+\widetilde{A_j^{reg}}.
\end{eqnarray*}

Clearly the operators $\mathcal{A}_j^1$ enjoy the mapping property $\mathcal{A}_j^1:L^2(\Gamma)\to H^1(\Gamma)$ and thus it can be seen that $\mathcal{S}^j:L^2(\Gamma)\to H^1(\Gamma)$, from which the claim of the lemma follows. Under the regularity assumption of the interface $\Gamma$, the arguments in the proof of the lemma carry over in the case of RtR operators corresponding to the impedance operators $Z_j^{PS}$, $j=0,1$.
\end{proof}

The result in Lemma~\ref{lemma1} is no longer valid in the case of Lipschitz interfaces $\Gamma$. To see this, we start from formula~\eqref{eq:S1_B1} and we get
  \begin{eqnarray*}
    \mathcal{S}^1&=&I-\alpha_1^{-1}(Z_0+Z_1)\mathcal{B}_1^{-1}S_1=(Z_0+Z_1)\mathcal{B}_1^{-1}\mathcal{B}_1^1(Z_0+Z_1)^{-1},\\
    \mathcal{B}_1^1&:=&\mathcal{B}_1-\alpha_1^{-1}S_1(Z_0+Z_1)=\mathcal{B}_1+2\alpha_1^{-1}S_1(\alpha_1N_{k1+i\sigma_1}+\alpha_0N_{k_0+i\sigma_0})\\
    &=&\mathcal{B}_1-\frac{1}{2}\alpha_1^{-1}(\alpha_0+\alpha_1)I + 2\alpha_1^{-1}(\alpha_0+\alpha_1)K_L^2+ \mathcal{B}_1^2,\\
    \mathcal{B}_1^2&:=&2\alpha_1^{-1}(\alpha_0+\alpha_1)(S_1-S_L)N_L+2\alpha_1^{-1}S_1(\alpha_1(N_{k1+i\sigma_1}-N_L)+\alpha_0(N_{k_0+i\sigma_0}-N_L)).
  \end{eqnarray*}
  We recall from the proof of Theorem~\ref{eq:inv_1} that the operator $\mathcal{B}_1$ was expressed in the form:
  \begin{eqnarray}
    \mathcal{B}_1&=&\mathcal{B}_{1,P}+\widetilde{\mathcal{B}}_1
  \end{eqnarray}
in terms of 
    \begin{eqnarray}
    \mathcal{B}_{1,P}&:=&\frac{1}{2}\alpha_1^{-1}(\alpha_0+\alpha_1)I+K_L-2\frac{\alpha_0}{\alpha_1}K^2_L
  \end{eqnarray}
and   the operator $\widetilde{\mathcal{B}_1}:L^2(\Gamma)\to H^1(\Gamma)$. Putting together these two representations we obtain
  \begin{eqnarray}
    \mathcal{B}_1^1=\mathcal{B}_{1,P}^1+\mathcal{B}_1^2+\widetilde{\mathcal{B}}_1\quad\mbox{where}\quad
    \mathcal{B}_{1,P}^1&:=&K_L+2K_L^2.\label{eq:b_1p}
  \end{eqnarray}
  Using the mapping properties recounted in Theorem~\ref{mapping} we see immediately that $\mathcal{B}_1^2:L^2(\Gamma)\to H^1(\Gamma)$. However, although both operators $\mathcal{B}_1^2$ and $\widetilde{\mathcal{B}}_1$ are compact in $L^2(\Gamma)$, the operator $\mathcal{B}_{1,P}^1$ is no longer compact in the same space, and hence the operator $\mathcal{B}_1^1$ is no longer compact in $L^2(\Gamma)$. Consequently, the RtR operator $\mathcal{S}^1$ is not compact either. Thus, one has to look deeper into the properties of the iteration operator $I-\mathcal{S}^0\mathcal{S}^1$. We thus have the following results:
\begin{theorem}\label{thm_Fredholm}
  In the case of Lipschitz interfaces $\Gamma$, the operators $I-\mathcal{S}^0\mathcal{S}^1$ are Fredholm of index zero in the space $L^2(\Gamma)$.
  \end{theorem}
  \begin{proof}
 We start with the assumption that $k_0$ is not a Laplace eigenvalue with Dirichlet boundary conditions in the domain $\Omega_1$. This allows us to use the representation of the RtR operators based on the operators $\mathcal{B}_j,j=0,1$, in which case the calculations are simpler. Using the splittings presented above we derive
  \begin{equation}\label{eq:f_diff_id}
  I-\mathcal{S}^0\mathcal{S}^1=(Z_0+Z_1)\mathcal{B}_0^{-1}(\mathcal{B}_0-\mathcal{B}_0^1\mathcal{B}_1^{-1}\mathcal{B}_1^1)(Z_0+Z_1)^{-1},
  \end{equation}
  where the operators $\mathcal{B}_0^1$ are defined analogously to the operators $\mathcal{B}_1^1$ by effecting similar decompositions to the operators $\mathcal{S}^0$. Clearly, we have that
  \begin{eqnarray*}
    \mathcal{B}_1\mathcal{B}_1^1=\mathcal{B}_{1,P}\mathcal{B}_{1,P}^1 +\mathcal{B}_R\quad\mbox{where}\quad
    \mathcal{B}_R:=\mathcal{B}_{1,P}(\mathcal{B}_1^2+\widetilde{\mathcal{B}}_1)+\widetilde{\mathcal{B}}_1\mathcal{B}_{1,P}^1,
  \end{eqnarray*}
  and
  \begin{eqnarray*}
   \mathcal{B}_1^1 \mathcal{B}_1=\mathcal{B}_{1,P}^1\mathcal{B}_{1,P}+\mathcal{B}_L\quad\mbox{where}\quad
    \mathcal{B}_L:=\mathcal{B}_{1,P}^1\widetilde{\mathcal{B}}_1+(\mathcal{B}_1^2+\widetilde{\mathcal{B}}_1)\widetilde{\mathcal{B}}_1.
  \end{eqnarray*}
  Since the operators $\mathcal{B}_{1,P}$ and $\mathcal{B}_{1,P}^1$ commute, it follows that
  \[
  \mathcal{B}_1\mathcal{B}_1^1-\mathcal{B}_1^1 \mathcal{B}_1=\mathcal{B}_R-\mathcal{B}_L:L^2(\Gamma)\to H^1(\Gamma).
  \]
  From the last identity we derive immediately
  \[
  \mathcal{B}_1^{-1}\mathcal{B}_1^1-\mathcal{B}_1^1 \mathcal{B}_1^{-1}=\mathcal{B}_1^{-1}(\mathcal{B}_L-\mathcal{B}_R)\mathcal{B}_1^{-1}:L^2(\Gamma)\to H^1(\Gamma).
  \]
  Using the last identity in equation~\eqref{eq:f_diff_id} we get that
  \[
  I-\mathcal{S}^0\mathcal{S}^1=(Z_0+Z_1)\mathcal{B}_0^{-1}(\mathcal{B}_0\mathcal{B}_1-\mathcal{B}_0^1\mathcal{B}_1^1)\mathcal{B}_1^{-1}(Z_0+Z_1)^{-1}+\mathcal{S}_R,
  \]
  where $\mathcal{S}_R:L^2(\Gamma)\to H^1(\Gamma)$. Using similar decompositions for the operator $\mathcal{B}_0$, a simple calculation delivers
  \begin{eqnarray*}
    \mathcal{B}_0\mathcal{B}_1-\mathcal{B}_0^1\mathcal{B}_1^1&=&\mathcal{B}_{0,P}\mathcal{B}_{1,P}-\mathcal{B}_{0,P}\mathcal{B}_{1,P}+\mathcal{B}^{reg},\\
    \mathcal{B}_{0,P}\mathcal{B}_{1,P}-\mathcal{B}_{0,P}\mathcal{B}_{1,P}&=&\frac{2(\alpha_0+\alpha_1)^2}{\alpha_0\alpha_1}\left(\frac{1}{2}I+K_L\right)^2\left(\frac{1}{2}I-K_L\right)
  \end{eqnarray*}
  and $\mathcal{B}^{reg}:L^2(\Gamma)\to H^1(\Gamma)$. Given that $\frac{1}{2}I+K_L$ is invertible in $L^2(\Gamma)$ and $\frac{1}{2}I-K_L$ is Fredholm of index zero in $L^2(\Gamma)$, it follows that $I-\mathcal{S}^0\mathcal{S}^1$ is also Fredholm of index zero in $L^2(\Gamma)$.

  The mechanics of the calculations above can be adapted to the case when the RtR operators $\mathcal{S}^j$ are represented via the operators $\mathcal{A}_j$. In this case, we make use of the splittings put forth in equations~\eqref{eq:A_comp} in the form:
  \begin{eqnarray*}
    \mathcal{A}_j&=&\mathcal{A}_{j,P}+\widetilde{\mathcal{A}}_j\\
    \mathcal{A}_{j,P}&:=&\alpha_j^{-1}(\alpha_j+\alpha_{j+1})I+\alpha_j^{-1}(\alpha_j-\alpha_{j+1})K_L-2\alpha_j^{-1}(\alpha_j+2\alpha_{j+1})K_L^2+4\alpha_j^{-1}\alpha_{j+1}K_L^3
  \end{eqnarray*}
  as well as
  \begin{eqnarray*}
    \mathcal{S}^j&=&(Z_0+Z_1)\mathcal{A}_j^{-1}\mathcal{A}_j^1(Z_0+Z_1)^{-1},\\
    \mathcal{A}_j^1&:=&\mathcal{A}_j-\alpha_j^{-1}(S_j+S_{\kappa_j}-2S_{\kappa_j}K_j^\top)(Z_0+Z_1)=\mathcal{A}_{j,P}+\mathcal{A}_j^{reg},\\
    \mathcal{A}_{j,P}&:=&2K_L+4K_L^2-4K_L^3,
  \end{eqnarray*}
  and $\mathcal{A}_j^{reg}:L^2(\Gamma)\to H^1(\Gamma)$. Just in the case of the calculations above pertaining to the use of operators $\mathcal{B}_j$, we can establish that
  \[
  I-\mathcal{S}^0\mathcal{S}^1=(Z_0+Z_1)\mathcal{A}_0^{-1}(\mathcal{A}_{0,P}\mathcal{A}_{1,P}-\mathcal{A}_{0,P}^1\mathcal{A}_{1,P}^1)\mathcal{A}_1^{-1}(Z_0+Z_1)^{-1}+\mathcal{S}^{reg}
  \]
  where $\mathcal{S}^{reg}:L^2(\Gamma)\to H^1(\Gamma)$. Now, we have that
  \[
  \mathcal{A}_{0,P}\mathcal{A}_{1,P}-\mathcal{A}_{0,P}^1\mathcal{A}_{1,P}^1=4\left(\frac{\alpha_0+\alpha_1}{\alpha_0\alpha_1}\right)(I-K_L)\left(\frac{1}{2}\ I+K_L\right),
  \]
  from which it follows that $I-\mathcal{S}^0\mathcal{S}^1$ is Fredholm of index zero in $L^2(\Gamma)$ for all real wavenumbers $k_j$, $j=0,1$. \end{proof}
  
We are now in the position to prove the main result:
\begin{theorem}\label{thm:wp_DDM}
  The DDM operators $I-\mathcal{S}^0\mathcal{S}^1:L^2(\Gamma)\to L^2(\Gamma)$ corresponding to the impedance operators $Z_j$ are invertible with continous inverses when the boundary $\Gamma$ is Lipschitz. In the case when the boundary $\Gamma$ is $C^2$, the DDM operators $I-\mathcal{S}^0\mathcal{S}^1:L^2(\Gamma)\to L^2(\Gamma)$ corresponding to the impedance operators $Z_j^{PS}$, $j=0,1$ are invertible with continous inverses.
\end{theorem}
\begin{proof}
  Given the results in Lemma~\ref{lemma1} and Theorem~\ref{thm_Fredholm}, it suffices to establish the injectivity of the DDM operator $I-\mathcal{S}^0\mathcal{S}^1$. The arguments in the proof hold for the regularity of the boundary $\Gamma$ stated in the hypothesis. Let $\varphi\in Ker(I-\mathcal{S}^0\mathcal{S}^1)$ and we consider the following Helmholtz equation
  \begin{eqnarray*}
    \Delta w_1+k_1^2w_1&=&0\qquad{\rm in}\ \Omega_1,\\
    \partial_{n_1}w_1+\alpha_1^{-1}Z_1w_1&=&\varphi\qquad{\rm on}\ \Gamma.
  \end{eqnarray*}
  Then, we have that
  \[
  \mathcal{S}^1\varphi=\partial_{n_1}w_1-\alpha_1^{-1}Z_0w_1.
  \]
  Consider also the following Helmholtz problem:
  \begin{eqnarray*}
    \Delta w_0+k_0^2w_0&=&0\qquad{\rm in}\ \Omega_0,\\
    \partial_{n_0}w_0+\alpha_0^{-1}Z_0w_0&=&\mathcal{S}^1\varphi\qquad{\rm on}\ \Gamma.
  \end{eqnarray*}
  and $w_0$ radiative at infinity. Using the fact that $\mathcal{S}^0\mathcal{S}^1\varphi=\varphi$ it  follows that 
  \[
  \mathcal{S}^0\mathcal{S}^1\varphi=\partial_{n_0}w_0-\alpha_0^{-1}Z_1w_0=\partial_{n_1}w_1+\alpha_1^{-1}Z_1w_1.
  \]
  Thus, we have derived the following system of equation on $\Gamma$
  \begin{eqnarray*}
    \partial_{n_0}w_0-\alpha_0^{-1}Z_1w_0&=&\partial_{n_1}w_1+\alpha_1^{-1}Z_1w_1,\\
    \partial_{n_0}w_0+\alpha_0^{-1}Z_0w_0&=&\partial_{n_1}w_1-\alpha_1^{-1}Z_0w_1,
  \end{eqnarray*}
  from which we get that
  \[
  (Z_0+Z_1)(\alpha_0^{-1}w_0+\alpha_1^{-1}w_1)=0\qquad{\rm on}\ \Gamma.
  \]
  Given the invertibility of the operator $Z_0+Z_1$ we obtain
  \[
  w_0|_\Gamma=-\alpha_1^{-1}\alpha_0 w_1|_\Gamma,
  \]
  and then 
  \[
  \partial_{n_1}w_0|_\Gamma=-\partial_{n_1}w_1|_\Gamma.
  \]
  Using the last two identities we derive
  \[
  \Im \int_\Gamma \overline{\partial_{n_1}w_0}\ w_0\ ds=\alpha_1^{-1}\alpha_0 \Im \int_\Gamma \overline{\partial_{n_1}w_1}\ w_1\ ds=\alpha_1^{-1}\alpha_0 \Im\int_{\Omega_1}(|\nabla w_1|^2-k_1^2w_1)dx=0.
  \]
  The last relation implies that $w_0=0$ identically in $\Omega_0$, from which follows immediately that $w_1=0$ in $\Omega_1$, and hence $\varphi=0$. 
  \end{proof}

We turn next to the case of DDM formulations with impedance operators $Z_j^{a},j=0,1$. The situation is quite different in this case due to the entirely different mapping properties of the operators $Z_j^{a},j=0,1$. Regarding this case we present the following result:
\begin{theorem}\label{thm:wp_DDM_a}
  The DDM operators $I-\mathcal{S}^0\mathcal{S}^1:L^2(\Gamma)\to H^1(\Gamma)$ corresponding to the impedance operators $Z_j^{a}$, $j=0,1$, are invertible with continous inverse when the boundary $\Gamma$ is $C^2$.
\end{theorem}
\begin{proof}
  We note that is suffices to establish the Fredholmness of the operators $I-\mathcal{S}^0\mathcal{S}^1:L^2(\Gamma)\to H^1(\Gamma)$. A key ingredient is to revisit the result established in formula~\eqref{eq:A_compS}, which in the case when the boundary $\Gamma$ is $C^2$ implies that
  \[
  \mathcal{A}_j=I+2\alpha_j^{-1}Z_j^aS_L+\widetilde{\mathcal{A}_{j,a}^{reg}},\quad j=0,1,
  \]
  where the operators $\widetilde{\mathcal{A}_{j,a}^{reg}}:L^2(\Gamma)\to H^{2}(\Gamma)$, and thus $\widetilde{\mathcal{A}_{j,a}^{reg}}:L^2(\Gamma)\to H^1(\Gamma)$ are compact for $j=0,1$. In the light of this fact, we obtain from formula~\eqref{eq:SjBI}
  \[
  \mathcal{S}^j=\mathcal{A}_j^{-1}(I-2\alpha_j^{-1}Z_{j+1}^aS_L)+\widetilde{\mathcal{S}}^j,\quad j=0,1,
  \]
  where $\widetilde{\mathcal{S}}^j:L^2(\Gamma)\to H^2(\Gamma)$. Following similar calculations to those in the proof of Theorem~\ref{thm_Fredholm} we arrive at
  \[
  I-\mathcal{S}^0\mathcal{S}^1=2(\alpha_0^{-1}+\alpha_1^{-1})(Z_0^a+Z_1^a)S_L+\mathcal{D},
  \]
  where $\mathcal{D}:L^2(\Gamma)\to H^2(\Gamma)$ and thus $\mathcal{D}:L^2(\Gamma)\to H^1(\Gamma)$ is a compact operator. Clearly, since $\Re(Z_j^a)>0,\ j=0,1$, the operator $I-\mathcal{S}^0\mathcal{S}^1$ satisfies a G\aa rding inequality given that $\Re \int_\Gamma S_L\varphi\ \overline{\varphi}\ ds\geq c\|\varphi\|^2_{H^{-1/2}(\Gamma)}$, and thus the operator $I-\mathcal{S}^0\mathcal{S}^1:L^2(\Gamma)\to H^1(\Gamma)$ is Fredholm of index zero. Its injectivity can be established by the same arguments as in the proof of Theorem~\ref{thm:wp_DDM}.
\end{proof}

To summarize, the well-posedness of the DDM formulations was established for all three choices of impedance operators in the case of $C^2$ boundaries $\Gamma$, and for the impedance operators $Z_j,j=0,1$ in the case of Lipschitz domains. 

\subsection{Domain decomposition approach with further subdomain divisions~\label{subdivisions}}

As it clear from Section~\ref{rtr}, the calculation of the RtR maps $\mathcal{S}^j,j=0,1$ needed in the DDM system~\eqref{DDM} requires operator inversions. In the high-frequency regime, the computation of the RtR maps requires inversion of large matrices, which becomes expensive if direct linear algebra solvers are employed. Furthermore, iterative Krylov subspace solvers require increasingly larger numbers of iterations for computation of interior RtR maps as the frequency increases, regardless of the BIE formulation used for these computations.  Thus, one possibility to reduce the computational costs incurred by the computation of RtR maps is to further subdivide the interior domain into a union of non-overlapping subdomains $\Omega_1=\cup_{j=1}^J \Omega_{1j}$. We assume that the decomposition is such that (1) each of the subdomains $\Omega_{1j}$ is simply connected/convex and (2) there are always cross points that belong to more than three subdomains---see Figure~\ref{fig:subdiv1} for an illustration of such subdivisions of an L-shaped domain $\Omega_1$. We define $\Gamma_{j\ell}:=\partial\Omega_{1j}\cap\partial\Omega_{1\ell}, 1\leq j,\ell$ in the case when the subdomains $\Omega_{1j}$ and $\Omega_{1\ell}$ share an edge in common, and $\Gamma_{j0}:=\partial\Omega_{1j}\cap\partial\Omega_0,\ 1\leq j$ in the case when the subdomain $\Omega_{1j}$ and $\partial\Omega_0=\Gamma$ share an edge in common. With these additional notations in place, the DDM system is written in the form
\begin{eqnarray}\label{DDMnew}
  \Delta u_{1j} +k_{1}^2 u_{1j} &=&0\qquad {\rm in}\ \Omega_{1j},\\
  \alpha_1\partial_{n_j}u_{1j}+Z_{1j}u_{1j}&=&-\alpha_1\partial_{n_\ell}u_{1\ell}+Z_{1j}u_{1\ell}\qquad{\rm on}\ \Gamma_{j\ell}\nonumber,\\
  \alpha_1\partial_{n_j}u_{1j}+Z_{1j} u_{1j}&=&-\alpha_0(\partial_{n_0}u_{0}+\partial_{n_0}u^{inc})+Z_{1j} (u_{0}+u^{inc})\qquad{\rm on}\ \Gamma_{j0},\nonumber\\
  \alpha_0(\partial_{n_0}u_{0}+\partial_{n_0}u^{inc})+Z_{0j} (u_{0}+u^{inc})&=&-\alpha_1\partial_{n_j}u_{1j}+Z_{0j} u_{1j}\qquad{\rm on}\ \Gamma_{j0}\nonumber,
\end{eqnarray}
where $n_j$ denotes the unit normal on $\partial\Omega_{1j}$ pointing to the exterior of the subdomain $\Omega_{1j}$. The transmission operators $Z_{0j}$ and $Z_{1j}$ in equations~\eqref{DDMnew} can be defined in the following way:
\begin{equation}\label{eq:first_localization}
Z_{0j}:=-2\alpha_1\chi_{0j}N_{\partial\Omega_{1j},k_1+i\sigma_1}\chi_{0j},
\end{equation}
where $\chi_{0j}$ is a smooth cutoff function supported on $\partial\Omega_{1j}\cap\partial\Omega_0$ (again, assumed to have non-zero one-dimensional measure), and respectively
\begin{equation}\label{eq:second_localization}
Z_{1j}:=-2\alpha_0\chi_{1j,0}N_{\Gamma,k_0+i\sigma_0}\chi_{1j,0}-2\alpha_1\sum_{\ell}\chi_{1j,1\ell}N_{\partial\Omega_\ell,k_1+i\sigma_1}\chi_{1j,1\ell},
\end{equation}
where $\chi_{1j,0}$ is a smooth cutoff function supported on $\partial\Omega_{1j}\cap\Gamma$ (in the case when the one-dimensional measure of the intersection is non-zero), and $\chi_{1j,1\ell}$ is a smooth cutoff function supported on $\Gamma_{j\ell}=\partial\Omega_{1j}\cap\partial\Omega_{1\ell}$ for all interior subdomains $\Omega_{1\ell},\ \ell\neq j$ that share an edge with the given subdomain $\Omega_{1j}$. This type of localization was previously discussed in~\cite{turc2016well}: the role of cutoffs is to blend the various operators in a manner that (1) is consistent for open arcs $\Gamma_{j\ell}$, and (2) gives rise to well-posed local Helmholtz problems. The cutoff technique also allows for blending of the Fourier multiplier operators $Z_j^{PS},j=0,1$. In the case of transmission operators $Z_j^a,j=0,1$, the blending is not necessary, since piece-wise constant impedances pose no difficulties. The DDM formulation~\eqref{DDMnew} can be recast in terms of computing Robin data:
\[
f_{1j}:=(\alpha_1\partial_{n_j}u_j+Z_{1j}u_j)|_{\partial\Omega_{1j}},\ j=1,\ldots,J\,
\]
and
\[
f_{0j}:=(\alpha_0\partial_{n_0}u_0+Z_{0j}u_0)|_{\partial\Omega_{1j}\cap\Gamma},\ meas(\partial\Omega_{1j}\cap\Gamma)\neq 0,
\]
via suitably defined RtR maps that take into account adjacent subdomains and their corresponding transmission operators. While the choice of blended transmission operators presented above gives rise to well-posed subdomain problems, the well-posedness of the DDM formulation~\eqref{DDMnew} remains an open question. Owing to the different nature of the transmission operators incorporated in the DDM formulation~\eqref{DDMnew}, the arguments used in the proof of the well-posedness of the DDM formulation~\eqref{DDM_t} that we presented in Section~\ref{DDMw} cannot be readily translated to the new setting.
\begin{figure}
\centering
\includegraphics[scale=1]{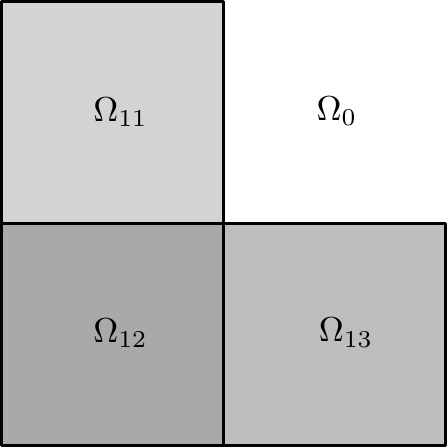}
\caption{Typical subdomain decomposition.}
\label{fig:subdiv1}
\end{figure}

 \section{Scalar transmission problems in partially coated domains\label{MS1}}

We consider next the problem of two dimensional transmission by structures that feature partial coatings, i.e.~penetrable scattering problems when parts of the boundary of the scatterer are perfectly conducting/impenetrable. Let $\Omega_1$ denote a bounded domain in $\mathbb{R}^2$ whose boundary $\Gamma:=\partial\Omega_1$ is given locally by the graph of a Lipschitz function, and let $\Omega_0:=\mathbb{R}^2\setminus\Omega_1$. We seek to find fields $u_0$ and $u_1$ that are solutions of the following scalar Helmholtz transmission problem:

\begin{equation}\label{system_tj}
\begin{array}{rclll}
 \Delta u_j +k_j^2 u_j &=&0& {\rm in}& \Omega_j,\quad j=0,1,\smallskip\\
  u_0+u^{inc}&=&u_1&{\rm on}& \Gamma_\TR,\smallskip\\
  \alpha_0(\partial_{n_0}u_0+\partial_{n_0}u^{inc})&=&-\alpha_1\partial_{n_1}u_1&{\rm on}& \Gamma_\TR,\smallskip\\
  \partial_{n_0}(u_0+u^{inc})&=&0&{\rm on}& \Gamma_{\PEC},\smallskip\\
  \partial_{n_1}u_1&=&0&{\rm on}& \Gamma_{\PEC},\smallskip\\
\multicolumn{4}{c}{\displaystyle \lim_{r\to\infty}r^{1/2}(\partial u_0/\partial r-ik_0u_0)=0,} \end{array}\end{equation}
where $\Gamma_\TR\cap \Gamma_{\PEC}=\Gamma$, and $|\Gamma_{\PEC}|>0$, where $|\Gamma_{PEC}|$ denotes the one dimensional measure of the set $\Gamma_{\PEC}$. We assume that the wavenumbers $k_j$ and the quantities $\alpha_j$ in the subdomains $\Omega_j$ are positive real numbers. The unit normal to the boundary $\partial\Omega_j$ is here denoted by $n_j$ and is assumed to point to the exterior of the subdomain $\Omega_j$. The incident field $u^{inc}$, on the other hand, is assumed to satisfy the Helmholtz equation with wavenumber $k_0$ in the unbounded domain $\Omega_0$. Finally, we assume that the parameters $\alpha_j$ are such that the transmission problem~\eqref{system_tj} is well posed.

\begin{figure}
\centering
\includegraphics[scale=1]{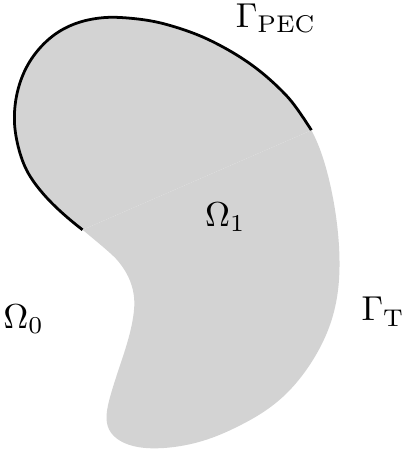}
\caption{Partially coated domain.}
\label{fig:coated}
\end{figure}

In what follows, we present again two formulations of the transmission problem~\eqref{system_tj}: a BIE  and a DDM approach.

\subsection{Boundary integral equation formulations\label{MS2}}

Again, boundary integral formulations of the transmission problem~\eqref{system_tj} can be derived using layer potentials defined on $\Gamma$: the solutions $u_j,j=0,1,$ of the transmission problem are sought in the form:
\begin{equation}\label{eq:layer_tj}
  u_j(\mathbf{x}):=SL_{\Gamma,j}\ v+(-1)^j\alpha_j^{-1} DL_{\Gamma,j}\ p,\quad \mathbf{x}\in\Omega_j,
\end{equation}
where $v$ and $p$ are densities defined on the $\Gamma$ and the double layer operators are defined with respect to exterior unit normals $\mathbf{n}$ corresponding to each domain $\Omega_j$. Here we used the same convention that the index $j$ in equation~\eqref{eq:layer_tj} refers to the wavenumber $k_j$ for $j=0,1$. The enforcement of the transmission conditions on the interface $\Gamma_T$ as well as the enforcement of the PEC conditions on the interface $\Gamma_{\PEC}$ lead to the following system of boundary integral equations:
\begin{equation}\label{eq:sk_system}
  \begin{array}{rclll}
  \frac{\alpha_0^{-1}+\alpha_1^{-1}}{2}\Pi_\TR p-\Pi_\TR(\alpha_0^{-1}K_{0}+\alpha_1^{-1}K_{1})p+\Pi_\TR(S_{1}-S_{0})v&=&u^{inc}&{\rm on}& \Gamma,\smallskip\\
  \frac{\alpha_0+\alpha_1}{2}\Pi_\TR v+\Pi_\TR(\alpha_0K_{0}^\top+\alpha_1K_{1}^\top)v+\Pi_\TR(N_{0}-N_{1})p&=&-\alpha_0\Pi_\TR\partial_{n_0}u^{inc}&{\rm on}&\Gamma,\smallskip\\
  \frac{1}{2}\Pi_{\PEC}v+\Pi_{\PEC}K_{0}^\top v+\alpha_0^{-1}\Pi_{\PEC}N_{0}p&=&-\Pi_{\PEC}\partial_{n_0}u^{inc}&{\rm on}& \Gamma,\smallskip\\
  \frac{1}{2}\Pi_{\PEC}v+\Pi_{\PEC}K_{1}^\top v-\alpha_1^{-1}\Pi_{\PEC}N_{1}p&=&0&{\rm on}& \Gamma,
  \end{array}
\end{equation}
where the restriction operators $\Pi_\TR$ and $\Pi_{\PEC}$ are defined as $\Pi_\TR\psi:=\psi|_{\Gamma_\TR}$ and respectively $\Pi_{\PEC}\psi:=\psi|_{\Gamma_{\PEC}}$, for functions $\psi$ defined on $\Gamma$. The restriction operators can be extended to distributions. In what follows we refer to the formulation~\eqref{eq:sk_system} by the acronym CFIESK. To the best of our knowledge, the well-posedness of the CFIESK formulations~\eqref{eq:sk_system} has not been established in the literature.

 \subsection{Domain decomposition approach\label{DDMcoat}}

DDM are natural candidates for numerical solution of transmission problems~\eqref{system_tj}. A non-overlapping domain decomposition approach for the solution of equations~\eqref{system_tj} consists of solving subdomain problems in $\Omega_j,j=0,1$ with matching Robin transmission boundary conditions on the common subdomain interface $\Gamma_T$. Indeed, this procedure amounts to computing the subdomain solutions:

\begin{eqnarray}\label{DDM}
  \Delta u_j +k_j^2 u_j &=&0\qquad {\rm in}\quad \Omega_j,\\
  \alpha_j(\partial_{n_j}u_j+\delta_j^0\partial_{n_j}u^{inc})+Z_j(u_j+\delta_j^0 u^{inc})&=&-\alpha_\ell(\partial_{n_\ell}u_\ell+\delta_\ell^0 \partial_{n_\ell}u^{inc})+Z_j (u_\ell+\delta_\ell^0 u^{inc})\quad{\rm on}\quad \Gamma_{\TR},\nonumber\\
  \partial_{n_j}u_j+\delta_j^0\partial_{n_j}u^{inc}&=&0\quad{\rm on}\quad \Gamma_{\PEC},\nonumber
\end{eqnarray}
where $\{j,\ell\}=\{0,1\}$ and $\delta_j^0$ stands for the Kronecker symbol, and $Z_j,Z_\ell$ are transmission operators that will be defined in what follows.  In order to describe the DDM method more concisely we introduce subdomain RtR maps~\cite{Collino1}. For each subdomain $\Omega_j,j=01,1$ we define RtR maps $\mathcal{S}^j,j=0,1$ in the following manner:
\begin{equation}\label{RtRboxj}
   \mathcal{S}^0(\psi_0):=(\alpha_0\partial_{n_0} u_0-Z_1u_0)|_{\Gamma_\TR}\quad\mbox{and}\quad \mathcal{S}^1(\psi_1):=(\alpha_1\partial_{n_1} u_1-Z_0u_1)|_{\Gamma_\TR}.
 \end{equation}
The DDM~\eqref{DDM} can be recast in terms of computing the global Robin data $f=[f_1\ f_0]^\top$ with
\[
f_{j}:=(\alpha_j\partial_{n_j}u_j+Z_j u_j)|_{\Gamma_\TR},\ j=0,1,
\]
as the solution of the following linear system that incorporates the subdomain RtR maps $\mathcal{S}^j,j=0,1$, previously defined
 \begin{equation}\label{ddm}
 (I+\mathcal{S})f=g\quad\mbox{where}\quad \mathcal{S}:=\begin{bmatrix}0&\mathcal{S}^0\\\mathcal{S}^1&0\end{bmatrix},
 \end{equation}
with right-hand side $g=[g_1\ g_0]^\top$ wherein
 \begin{eqnarray}\label{rhs_ddm}
   g_1&=& (-\alpha_0\partial_{n_0}u^{inc}+Z_1u^{inc})|_{\Gamma_\TR},\nonumber\\
   g_0&=&-(\alpha_0\partial_{n_0}u^{inc}+Z_0u^{inc})|_{\Gamma_\TR}.\nonumber
   \end{eqnarray}
Ideally, the operator $Z_0$ should be the restriction to $\Gamma_\TR$ of Dirichlet-to-Neumann (DtN) operator corresponding to the Helmholtz equation posed in the domain $\Omega_1$ with generalized Robin boundary conditions on $\Gamma_\TR$ and zero Neumann boundary conditions on $\Gamma_{\PEC}$. Following the methodology presented in the context of classical transmission problems, we employ a smooth cutoff function $\chi_\TR$ supported on $\Gamma_\TR$ in order to define the following the transmission operators:
 \begin{equation}\label{eq:calT}
 Z_0=-2\alpha_0\chi_{\TR}N_{\Gamma,k_1+i\sigma_1}\chi_{\TR}\quad\mbox{and}\quad Z1 = -2\alpha_1\chi_{\TR}N_{\Gamma,k_0+i\sigma_0}\chi_{\TR},\quad \sigma_j>0.
 \end{equation}
We can also use principal symbol transmission operators define accordingly
 \begin{equation}\label{eq:calPST}
 Z_0^{PS}=-2\alpha_0\chi_{\TR}PS(N_{\Gamma,k_1+i\sigma_1})\chi_{\TR}\quad\mbox{and}\quad Z_1^{PS} = -2\alpha_1\chi_{\TR}PS(N_{\Gamma,k_0+i\sigma_0})\chi_{\TR},\quad \sigma_j>0,
 \end{equation}
 as well as the simpler transmission operators 
 \begin{equation}\label{eq:dampingT}
   Z_0^a=-i\alpha_0(k_1+i\sigma_1)\Pi_\TR\quad\mbox{and}\quad Z_1^a=-i\alpha_1(k_0+i\sigma_0)\Pi_\TR.
 \end{equation}
 The RtR maps corresponding to the new transmission operators defined in equations~\eqref{eq:calT},~\eqref{eq:calPST}, and~\eqref{eq:dampingT} can be computed by readily incorporating in the methodology presented in Section~\ref{rtr} the additional requirement of zero Neumann traces on the portion $\Gamma_{\PEC}$ of the boundaries. 

\section{High-order Nystr\"om discretizations~\label{Nystrom}}

\parskip 10pt plus2pt minus1pt
\parindent0pt 

We use Nystr\"om discretizations of  the CFIESK equation~\eqref{eq:sk_system}, as well as the RtR maps associated with the various DDM formulations. The key ingredient is the Nystr\"om discretization of the four BIO in the Calder\'on calculus for piecewise smooth boundaries. These discretizations were introduced in~\cite{dominguez2016well} where this methodology was presented in full detail. In particular, the discretization of the CFIESK equation~\eqref{eq:sk_system} was described in the aforementioned contribution. Therefore, we present here the discretization of the DDM formulations that relies, in turn, on discretizations of the corresponding RtR maps. Specifically, graded meshes produced by means sigmoid transforms~\cite{KressCorner} that accumulate points polynomially toward corner and junction points (where $\Gamma_{PEC}$ and $\Gamma_T$ meet) are utilized on the closed curve $\Gamma$.  For each of the subdomains $\Omega_j$, $j=0,1$, we employ graded meshes denoted by  $$L_j:=\{\mathbf{x}^j_{m},m=0,\ldots,N_j-1\}\quad{\rm on}\quad \partial \Omega_j=\Gamma,$$ with the same polynomial degree of the sigmoid transforms on all subdomains. All meshes in the parameter space $[0,2\pi]$ are shifted by the same amount so that none of the grid points on the skeleton corresponds to a triple/multiple junction or a corner point. We allow for non-conforming meshes, that is $N_1$ may not be equal to $N_0$; the size $N_j$ of the mesh $L_j$ is chosen to resolve the wavenumber $k_j$ corresponding to the domain $\Omega_j$.  

Using graded meshes that avoid corner points, trigonometric interpolation, and the classical singular quadratures of Kusmaul and Martensen~\cite{kusmaul,martensen}, we perform the Nystr\"om discretization presented in~\cite{dominguez2016well} to produce high-order $ N_j\times N_j$ collocation matrix approximations of the four BIO described in equations~\eqref{traces}. We note that discretizations of the Fourier multiplier operators $Z_j^{PS}$, $j=0,1$ is straightforward via trigonometric interpolation~\cite{dominguez2016well}. Based on these, the DDM algorithm proceeds with a precomputational stage whereby matrix approximations of all the RtR maps needed are produced. The precomputational stage is computationally expensive on account of the matrix inversions needed for the computation of discrete RtR matrices. Nevertheless, this stage is highly parallelizable since the computation of the RtR matrix corresponding to a subdomain does not require information from adjacent subdomains. In order to avoid complications related to singularities at junction/cross points, we replace in the DDM algorithm the RtR maps by {\em weighted} parametrized counterparts
\[
\mathcal{S}^{j,w}(\alpha_j|\mathbf{x}_j'|\partial_{n_j}u_j+Z_j\ u_j):=\alpha_j|\mathbf{x}_j'|\partial_{n_j}u_j-Z_{j+1}\ u_j.
\]
Collocated discretizations of the latter weighted RtR maps can be easily computed through a simple modification of the methodology introduced in~\cite{turc2016well} and recounted above. Nevertheless, the representation of RtR maps in terms of BIO requires use of inverses of matrices corresponding to Nystr\"om discretizations of either operators $\mathcal{B}_j$, \emph{cf.}~\eqref{eq:S1_B1}, $\mathcal{A}_j$, \emph{cf.}~\eqref{eq:SjBI}, or $\mathcal{C}_j$, \emph{cf.}~\eqref{eq:Steinbach}. The inversion of these matrices can be performed via direct or iterative linear algebra methods. In the former case, the discretization of the weighted RtR maps corresponding to each domain $\partial \Omega_j$ is constructed as $N_j\times N_j$ collocation matrices  $\mathcal{S}^j_{N_j}$. For bounded (interior) domains, the formulations based on the use of the simpler operators $\mathcal{B}_1$ are the most efficient for use of direct linear algebra solvers; the ones based on operators $\mathcal{A}_1$ are more complex, and the ones based on the operators $\mathcal{C}_1$ require inversions of matrices twice as large. For the unbounded domain $\Omega_0$, the formulations based on the use of operators $\mathcal{A}_0$ are preferred owing to their stability valid for all real wavenumbers $k_0$. However, the use of direct linear algebra solvers at this stage imposes limitations on the discretization size $N_j$. This size can be further reduced by employing subdivisions of the interior domain $\Omega_1$ as described in Section~\ref{subdivisions}. For the examples considered in this text, such subdivisions are straightforward. Alternatively, when iterative linear algebra methods are employed for the calculation of RtR maps, the latter are an inner iteration in the iterative solution of the DDM linear system~\eqref{ddm_t}. Our numerical experiments presented in the next Section suggest that the use of BIE formulations based on the operators $\mathcal{A}_0$~\eqref{eq:SjBI} for the calculation of exterior RtR maps $\mathcal{S}^0$ result in small numbers of iterations that grow slowly as the frequency increases. The situation is entirely different in the case of interior RtR maps $\mathcal{S}^1$: all three BIE formulations considered in this text give rise to numbers of iterations that grow significantly with the frequency. Again, a remedy for this issue is employing subdivisions of the interior domain $\Omega_1$ and thus effectively reducing the acoustic/electric size of the subdomains.

Once the discretized RtR matrices $S^0_{N_0}\in \mathbb{R}^{N_0\times N_0}$ and respectively $S^1_{N_1}\in \mathbb{R}^{N_1\times N_1}$ are computed (we assume in what follows that $k_0<k_1$ and thus $N_0\leq N_1$) the discretization of the DDM linear system~\eqref{ddm_t} is easily set up in the form
\begin{eqnarray}\label{ddm_t_discrete}
  f_0^{N_0} + P_{N_1\to N_0}\mathcal{S}^1_{N_1}f_1^{N_1}&=&g_0^{N_0},\nonumber\\
  f_1^{N_1} +E_{N_0\to N_1}\mathcal{S}^0_{N_0}f_0^{N_0}&=&g_1^{N_1},
 \end{eqnarray}
where $f_j^{N_j}$ are approximations of the Robin data $f_j$ trigonometrically collocated on the grids $L_j$ for $j=0,1$, and the projection operator $P_{N_1\to N_0}$ and the extension operator $E_{N_0\to N_1}$ allow for transfer of information via Fourier space from the two grids $L_0$ and $L_1$. Specifically, the extension operator $E_{N_0\to N_1}$ is realized via zero padding in the Fourier space, while the projection operator $P_{N_1\to N_0}$ is a cutoff operator in the Fourier space. The right hand-side in equation~\eqref{ddm_t_discrete} are obtained by simply evaluating $g_j$ on the grids $L_j$ for $j=0,1$. In order to further reduce the size of the linear system that we solve, we further eliminate the data $f_1^{N_1}$ from the linear system~\eqref{ddm_t_discrete} and solve the reduced linear system
\begin{equation}\label{eq:ddm_final}
  f_0^{N_0}-P_{N_1\to N_0}\mathcal{S}^1_{N_1}E_{N_0\to N_1}\mathcal{S}^0_{N_0}f_0^{N_0}=g_0^{N_0}-P_{N_1\to N_0}\mathcal{S}^1_{N_1}g_1^{N_1}.
  \end{equation}
Once the exterior Robin data $f_0^{N_0}$ is computed by solving the linear system~\eqref{eq:ddm_final}, the exterior Cauchy data on $\Gamma$ can be immediately retrieved via the RtR operator $\mathcal{S}^0$. The interior Cauchy data on $\Gamma$ is then readily computed from the continuity conditions. In what follows, we present a concise algorithmic description of the DDM formulation~\eqref{ddm_t}. The modifications needed to cover the DDM with further domain subdivisions~\eqref{DDMnew} or the DDM for partial coatings~\eqref{ddm} are straightforward.

\begin{algorithm}
  \SetAlgoLined
  Offline: For each subdomain $\Omega_j$, discretize all the BIO that feature in formulas~\eqref{eq:int_D} and~\eqref{eq:CFIER2} corresponding to each boundary $\partial\Omega_j$ using Nystr\"om discretizations. The discretization of each BIO results in a collocation matrix of size $N_j\times N_j$, whose computational cost is $\mathcal{O}(N_j^2)$\;
  Offline: Compute all the collocated subdomain RtR matrices $\mathcal{S}^j_{N_j}$ using formulation~\eqref{eq:int_D} for the interior domain and the formulation~\eqref{eq:CFIER2} for the exterior domain. We compute discretizations of the RtR maps via LU factorizations, and thus the cost of evaluating each subdomain RtR map is $\mathcal{O}(N_j^3)$\;
  Solution: Set up the DDM linear system according to formula~\eqref{eq:ddm_final} and solve for the Robin data $f_0^{N_0}$ using GMRES\;
  Post-processing: Use the Robin data $f_0^{N_0}$ computed in the previous step and the RtR matrix $\mathcal{S}^0_{N_0}$ to compute Cauchy data on $\Gamma$.
 \caption{Description of the DDM algorithm}{\label{DDM1}}
\end{algorithm}

 \section{Numerical results\label{num}}

 In this section we present numerical experiments concerning the iterative behavior of various DDM solvers considered in this text. We also document the iterative behavior of the CFIESK solvers. We mention that a comprehensive comparison between various integral formulations for transmission problems was pursued in~\cite{turc2,dominguez2016well,jerez2017multitrace}. While the CFIESK formulations are not the most performant formulations vis-a-vis iterative solvers, they are the simplest and most widely used in the collocation discretization community~\cite{rokhlin-dielectric,greengard1}. Also, and as mentioned previously, the CFIESK can be relatively easily extended to more complex boundary conditions scenarios. It is not our goal to carry in this text a detailed computational efficiency comparison between BIE formulations and DDM formulations of Helmholtz transmission problems. On the one hand, there is a relatively large body of work in which fast methods and matrix compression techniques are used to accelerate the performance of BIE based solvers~\cite{greengard1,br-turc}. On the other hand, DDM with quasi-optimal transmission operators for transmission Helmholtz equations have been studied to a very limited extent; it is our intent to highlight in this paper the remarkable iterative properties that these solvers enjoy, and to point out several challenges that they face related to efficient computations of RtR maps. It is important to bear in mind that one of the main attractive feature of DDM is their embarassing parallelism, which is much harder to achieve by BIE solvers. We plan to pursue elsewhere an in depth comparison between the computational efficiency of BIE solvers and DDM solvers for three dimensional transmission problems.   

 All of the formulations considered were discretized following the prescription in Section~\ref{Nystrom}. In all the numerical experiments we used meshes that rely on sigmoid transforms of polynomial degree 3. Also,  following the optimality prescriptions in~\cite{boubendirDDM}, we selected $\sigma_j=k_j^{1/3}$ in the definition of the complex wavenumbers that enter the definition of the corresponding tranmsission operators. Unless specified otherwise, in all the numerical experiments we present numbers of GMRES iterations for various solvers to reach a relative residual of $10^{-4}$ and present errors in the far-field for $1024$ equi-spaced far-field directions. In all the numerical results presented, the reference solutions were computed using highly refined discretizations of CFIESK solvers. We start in Table~\ref{comp3} with an illustration of the accuracy of the Nystr\"om discretizations of the CFIESK and various DDM formulations of the transmission problem~\eqref{system_t} that used conforming meshes, that is $N_0=N_1$. We note that the CFIESK and DDM with transmission operators $Z_j$ and $Z_j^{PS}$ exhibit iterative behaviors corresponding to second kind formulations, while the DDM with transmission operators $Z_j^a$ behave like first kind formulations. Also, the solvers based on CFIESK formulations are more accurate than the DDM solvers, and the accuracy of the latter formulations is virtually independent of the choice of transmission operators. 

\begin{table}
   \begin{center}
     \resizebox{!}{1.2cm}
{   
\begin{tabular}{|c|c|c|c|c|c|c|c|c|}
\hline
Unknowns & \multicolumn{2}{c|} {CFIESK} & \multicolumn{2}{c|} {DDM $Z_j,j=0,1$} &  \multicolumn{2}{c|} {DDM $Z_j^{PS},j=0,1$} &\multicolumn{2}{c|} {DDM $Z_j^a,j=0,1$}\\
\cline{2-9}
& It & $\varepsilon_\infty$ &It & $\varepsilon_\infty$ & It & $\varepsilon_\infty$ & It & $\varepsilon_\infty$ \\
\hline
72 & 51 & 9.2 $\times$ $10^{-4}$ & 26 & 4.3 $\times$ $10^{-3}$ & 30 & 4.3 $\times$ $10^{-3}$ & 54 & 4.3 $\times$ $10^{-3}$\\
144 & 51 & 5.6 $\times$ $10^{-6}$ & 26 &  3.4 $\times$ $10^{-4}$ & 30 & 3.4 $\times$ $10^{-4}$ & 66 & 3.4 $\times$ $10^{-4}$\\
288 & 51 & 3.9 $\times$ $10^{-7}$ & 26 & 3.9 $\times$ $10^{-5}$ & 30 & 3.9 $\times$ $10^{-5}$ & 74 & 3.9 $\times$ $10^{-5}$\\
572 & 51 & 2.5 $\times$ $10^{-8}$ & 25 & 4.1 $\times$ $10^{-6}$ & 30 & 4.1 $\times$ $10^{-6}$ & 87 & 4.1 $\times$ $10^{-6}$\\
1144 & 51 & 1.6 $\times$ $10^{-9}$ & 25 & 2.6 $\times$ $10^{-7}$ & 30 & 2.6 $\times$ $10^{-7}$ & 104 & 2.6 $\times$ $10^{-7}$\\
\hline
\end{tabular}
}
\caption{Far-field errors $\varepsilon_\infty$ computed using various formulations considered in this text in the case of scattering from an L-shaped domain with $\omega=2$, $\varepsilon_0=1$, and $\varepsilon_1=4$ with $\alpha_j=1, j=0,1$. We considered a GMRES residual of $10^{-12}$ in all the tests presented in the Table. CFIESK formulations uses twice as many unknowns as the DDM formulations.\label{comp3}}
\end{center}
 \end{table}
 
We present in Tables~\ref{comp4} and~\ref{comp5} the behavior of the various formulations for the transmission problem~\eqref{system_t} as a function of frequency in the case of high-contrast material properties, that is $\varepsilon_0=1$ and $\varepsilon_1=16$ and two scatterers: a square of size 4 in Table~\ref{comp4} and an L-shaped domain of size 4 in Table~\ref{comp5}. We used conforming meshes, i.e. $N_0=N_1$ for the DDM solvers. As it can be seen from the results in Tables~\ref{comp4} and~\ref{comp5}, the numbers of iterations required by the DDM solvers with transmission operators $Z_j$, $j=0,1$ are small and depend very mildly on the increasing frequency. Also, the iterative behavior of the DDM based on the transmission operators $Z_j^{PS}$, $j=0,1$, deteriorates somewhat with respect to that of DDM solvers with transmission operators $Z_j$, $j=0,1$. In contrast, the iterative behavior of DDM based on the simplest transmission operators $Z_j^a,j=0,1$ is quite poor in the high-frequency, high-contrast case. 
   
  \begin{table}
   \begin{center}
     \resizebox{!}{1.2cm}
{   
\begin{tabular}{|c|c|c|c|c|c|c|c|c|}
\hline
$\omega$ & \multicolumn{2}{c|} {CFIESK} & \multicolumn{2}{c|} {DDM $Z_j,j=0,1$} &  \multicolumn{2}{c|} {DDM $Z_j^{PS},j=0,1$} &\multicolumn{2}{c|} {DDM $Z_j^a,j=0,1$}\\
\cline{2-9}
& It & $\varepsilon_\infty$ &It & $\varepsilon_\infty$ & It & $\varepsilon_\infty$ & It & $\varepsilon_\infty$ \\
\hline
1 & 24 & 3.1 $\times$ $10^{-4}$ & 10 & 5.2 $\times$ $10^{-3}$ & 10 & 5.1 $\times$ $10^{-3}$ & 20 & 5.0 $\times$ $10^{-3}$\\
2 & 39 & 8.2 $\times$ $10^{-4}$ & 11 &  1.0 $\times$ $10^{-3}$ & 12 & 9.9 $\times$ $10^{-4}$ & 28 & 1.1 $\times$ $10^{-3}$\\
4 & 93 & 2.3 $\times$ $10^{-3}$ &  12 & 1.2 $\times$ $10^{-3}$ & 17 & 1.4 $\times$ $10^{-3}$ & 46 & 1.3 $\times$ $10^{-3}$\\
8 & 162 & 6.3 $\times$ $10^{-3}$ & 10 & 2.1 $\times$ $10^{-3}$ & 19 & 2.2 $\times$ $10^{-3}$ & 84 & 2.1 $\times$ $10^{-3}$\\
16 & 333 & 7.6 $\times$ $10^{-3}$ & 11 & 4.5 $\times$ $10^{-3}$ & 29 & 4.2 $\times$ $10^{-3}$ & 151 & 4.1 $\times$ $10^{-3}$\\
32 & 565 & 1.2 $\times$ $10^{-2}$ & 13  & 2.9 $\times$ $10^{-3}$ &  56 & 2.8 $\times$ $10^{-3}$ & 253 & 2.9 $\times$ $10^{-3}$\\
\hline
\end{tabular}
}
\caption{Far-field errors $\varepsilon_\infty$ computed using various formulations considered in this text in the case of scattering from a square of size 4 with $\varepsilon_0=1$ and $\varepsilon_1=16$ with $\alpha_j=1, j=0,1$. The DDM discretization used conforming meshes, that is $N_0=N_1$, and $64, 128, 256, 512, 1024$ and respectively $2048$ unknonws (these are the values of $N_0$); CFIESK formulations used twice as many unknowns. The numbers of iterations required by the DDM solvers with transmission operators $Z_j$, $j=0,1$, were 13, 15, 14, 19, 23, and respectively 31 in the case when $\alpha_j=\varepsilon_j^{-1}, j=0,1$.\label{comp4}}
\end{center}
 \end{table}

\begin{table}
   \begin{center}
     \resizebox{!}{1.2cm}
{   
\begin{tabular}{|c|c|c|c|c|c|c|c|c|}
\hline
$\omega$ & \multicolumn{2}{c|} {CFIESK} & \multicolumn{2}{c|} {DDM $Z_j,j=0,1$} &  \multicolumn{2}{c|} {DDM $Z_j^{PS},j=0,1$} &\multicolumn{2}{c|} {DDM $Z_j^a,j=0,1$}\\
\cline{2-9}
& It & $\varepsilon_\infty$ &It & $\varepsilon_\infty$ & It & $\varepsilon_\infty$ & It & $\varepsilon_\infty$ \\
\hline
1 & 43 & 1.0 $\times$ $10^{-3}$ & 15 & 4.7 $\times$ $10^{-3}$ & 16 & 4.6 $\times$ $10^{-3}$ & 31 & 4.6 $\times$ $10^{-3}$\\
2 & 72 & 1.1 $\times$ $10^{-3}$ & 15 &  9.0 $\times$ $10^{-4}$ & 17 & 1.2 $\times$ $10^{-3}$ & 46 & 8.3 $\times$ $10^{-4}$\\
4 & 135 & 2.1 $\times$ $10^{-3}$ & 16 & 2.4 $\times$ $10^{-3}$ & 24 & 2.4 $\times$ $10^{-3}$ & 81 & 2.3 $\times$ $10^{-3}$\\
8 & 208 & 2.4 $\times$ $10^{-3}$ & 15 & 4.0 $\times$ $10^{-3}$ & 29 & 4.0 $\times$ $10^{-3}$ & 112 & 4.1 $\times$ $10^{-3}$\\
16 & 493 & 8.8 $\times$ $10^{-3}$ & 21 & 8.1 $\times$ $10^{-3}$ & 56 & 8.1 $\times$ $10^{-3}$ & 276 & 8.0 $\times$ $10^{-3}$\\
32 & 887 & 1.2 $\times$ $10^{-2}$ & 22 & 9.6 $\times$ $10^{-3}$ &  87 & 9.6 $\times$ $10^{-3}$ & 488 & 9.6 $\times$ $10^{-3}$\\
\hline
\end{tabular}
}
\caption{Far-field errors computed using various formulations considered in this text in the case of scattering from a L-shaped domain of size 4 with $\varepsilon_0=1$ and $\varepsilon_1=16$ with $\alpha_j=1, j=0,1$. The DDM discretization used conforming meshes, that is $N_0=N_1$ and $64, 128, 256, 512, 1024$ and respectively $2048$ unknonws (these are the values of $N_0$); CFIESK formulations used twice as many unknowns. The numbers of iterations required by the DDM solvers with transmission operators $Z_j,j=0,1$ were 21, 23, 21, 23, 29, and respectively 37 in the case when $\alpha_j=\varepsilon_j^{-1}, j=0,1$.\label{comp5}}
\end{center}
\end{table}

The superior iterative performance of the DDM formulations that rely on transmission operators $Z_j,j=0,1$ can be inferred from the clustering of the eigenvalues of the iteration operator $I-\mathcal{S}^1\mathcal{S}^0$ around one. We present in Figure~\ref{fig:eig} the remarkable eigenvalue clustering in the case of the L-shaped scatterer for high-frequencies. It is important to note from the evidence presented in Figure~\ref{fig:eig} that although the eigenvalues of the iteration operator $I-\mathcal{S}^1\mathcal{S}^0$ corresponding to high-frequency eigenmodes are tighlty clustered around one, the operator $\mathcal{S}^1\mathcal{S}^0$ is not a contraction. 

\begin{figure}
\centering
\includegraphics[height=60mm]{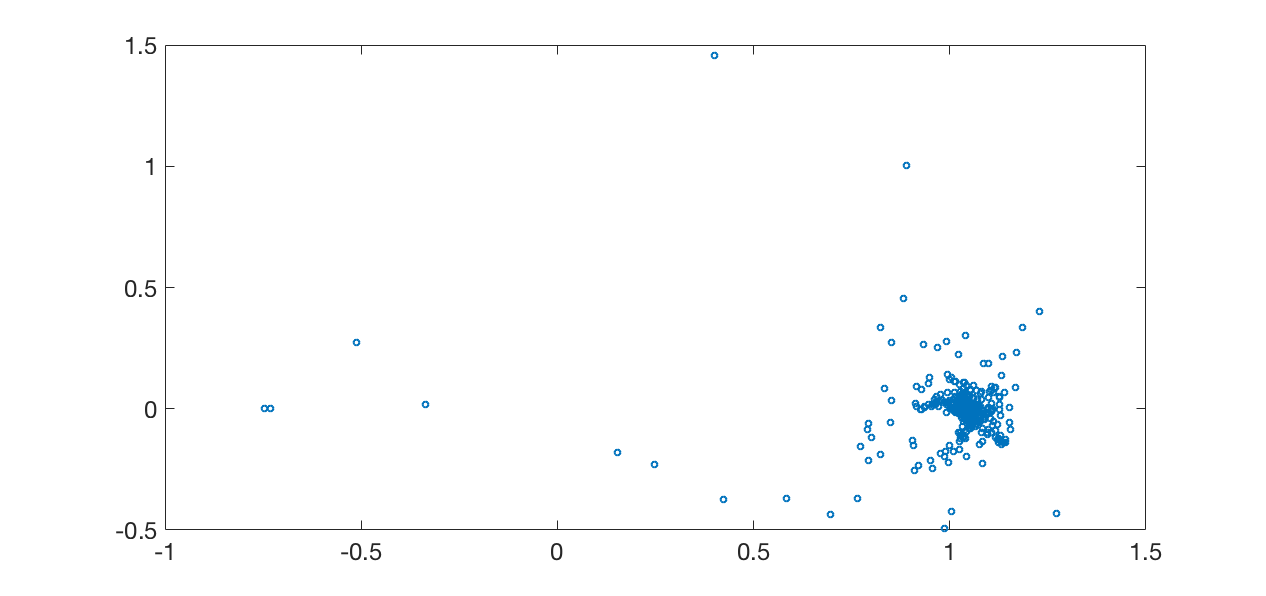}\\
\includegraphics[height=60mm]{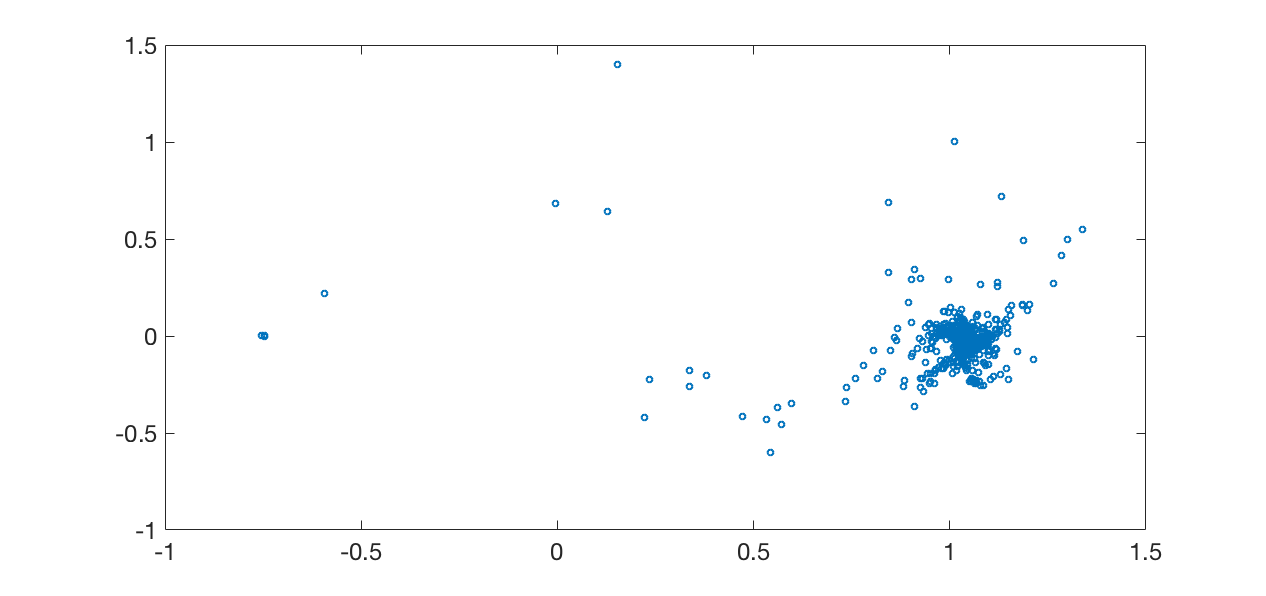}
\caption{Eigenvalue distributions of the DDM iteration operator $I-\mathcal{S}^1\mathcal{S}^0$ with the choice of transmission operators $Z_j,j=0,1$ for the L-shaped scatterer and high-contrast transmission problems with $\varepsilon_0=1$, $\varepsilon_1=16$, $\alpha_j=1, j=0,1$ and $\omega=16$ (top) and $\omega=32$ (bottom).}
\label{fig:eig}
\end{figure}

Clearly, in the case of high-frequency, high-contrast transmission problems, DDM that use conforming meshes are not the most advantageous computationally. Rather, the use of non-conforming meshes that resolve the wavenumber corresponding to each subdomain are more favorable. We present in Table~\ref{comp51} results corresponding to use of non-conforming meshes in the DDM with transmission operators $Z_j$, $j=0,1$. We note that the iterative behavior of the non-conforming DDM is very similar to that of conforming DDM, without major compromise on accuracy.   

\begin{table}
   \begin{center}
     \resizebox{!}{1.0cm}
{   
\begin{tabular}{|c|c|c|c|c|c|c|c|c|c|c|c|c|c|}
\hline
$\omega$ & \multicolumn{3}{c|} {DDM (1) $Z_j,j=0,1$ Square } & \multicolumn{3}{c|} {DDM (2) $Z_j,j=0,1$ Square } &  \multicolumn{3}{c|} {DDM (1) $Z_j,j=0,1$ L-shape} &\multicolumn{3}{c|} {DDM (1) $Z_j,j=0,1$ L-shape}\\
\cline{2-13}
& $N_0=N_1$ & It & $\varepsilon_\infty$ & $N_0$ & It & $\varepsilon_\infty$ & $N_0=N_1$ & It & $\varepsilon_\infty$ & $N_0$ & It & $\varepsilon_\infty$ \\
\hline
4 & 256 & 10 & 1.2 $\times$ $10^{-3}$ & 192 & 10 & 1.2 $\times$ $10^{-3}$ & 256 & 16 & 2.4 $\times$ $10^{-3}$ & 192 & 14 & 6.0 $\times$ $10^{-3}$\\
8 & 512 & 10 & 2.1 $\times$ $10^{-3}$ & 384 & 14 & 6.1 $\times$ $10^{-3}$ & 512 & 15 & 4.0 $\times$ $10^{-3}$ & 384 & 12 & 3.1 $\times$ $10^{-3}$\\
16 & 1024 & 11 & 4.5 $\times$ $10^{-3}$ & 768 & 16 & 6.7 $\times$ $10^{-3}$ & 1024 & 21 & 8.1 $\times$ $10^{-3}$ & 768 & 22 & 1.2 $\times$ $10^{-2}$\\
32 & 2048 & 13 & 2.9 $\times$ $10^{-3}$ &  1536 & 15 & 4.9 $\times$ $10^{-3}$ & 2048 & 22 & 9.6 $\times$ $10^{-3}$ & 1536 & 27 & 1.3 $\times$ $10^{-2}$\\
\hline
\end{tabular}
}
\caption{Comparison between the conforming ($N_0=N_1$) and non-conforming ($N_0<N_1$) DDM with transmission operators $Z_j,j=0,1$ for  high-contrast transmission problems with $\varepsilon_0=1$ and $\varepsilon_1=16$ with $\alpha_j=1, j=0,1$. In the non-conforming case, the values of $N_1$ are equal to those in the conforming case for the same frequency.\label{comp51}}
\end{center}
\end{table}

The use of optimized transmission operators $Z_j$ and $Z_j^{PS}$ for $j=0,1$ gives rise to superior DDM iterative performance. However, given that the transmission operators $Z_j$, $j=0,1$, and $Z_j^{PS}$, $j=0,1$ are non-local operators, their implementation favors boundary integral equation solvers, while posing challenges to finite difference/finite element discretizations. Therefore, approximations of the square root Fourier multiplier operators $Z_j^{PS}$ more amenable to the latter types of discretizations were proposed in the literature. There are two classes of such approximations that were widely used: local second order approximations with optimized coefficients~\cite{Gander1} and Pad\'e approximations. Reference~\cite{boubendirDDM} provides numerical evidence that the incorporation of Pade\'e approximations of square root operators results in DDM with faster rates of convergence than the use of local second order approximations. In what follows, we explain briefly the Pad\'e approximations used in~\cite{boubendirDDM}; we start from formulas
\[
\sqrt{1+X}\approx e^{i\theta/2}R_p(e^{-i\theta}X)=A_0+\sum_{j=1}^p\frac{A_jX}{1+B_jX}
\]
where the complex numbers $A_0$, $A_j$ and $B_j$ are given by
\[
A_0=e^{i\theta/2}R_p(e^{-i\theta}-1),\quad A_j=\frac{e^{-i\theta/2}a_j}{(1+b_j(e^{-i\theta}-1))^2},\quad B_j=\frac{e^{-i\theta}b_j}{1+b_j(e^{-i\theta}-1)}
\]
and
\[
R_p(z)=1+\sum_{j=1}^p\frac{a_j z}{1+b_j z}
\]
with
\[
a_j=\frac{2}{2p+1}\sin^2(\frac{j\pi}{2p+1})\qquad b_j=\cos^2(\frac{j\pi}{2p+1}).
\]
These Pad\'e approximations of square roots above give rise to the following transmission operators
\begin{equation}\label{eq:Pade}
  Z_j^{Pade,p}=-\frac{i}{2}(k_j+i\sigma_j)\left(A_0I-\sum_{j=1}^pA_j\left(\frac{\partial_s^2}{(k_j+i\sigma_j)^2}\right)\left(I-B_j\left(\frac{\partial_s^2}{(k_j+i\sigma_j)^2}\right)\right)^{-1}\right),
\end{equation}
where $\partial_s$ is the tangential derivative on $\Gamma$. We note that the discretizations of the operators $Z_j^{Pade,p},j=0,1$ defined in equation~\eqref{eq:Pade} is relatively straightforward using trigonometirc interpolants. However, their discretization requires $p$ matrix inverses per wavenumber. We present in Figure~\ref{fig:pade} a comparison between the DDM iterations as a function of the Pad\'e parameter $p$ in the case of a L-shaped scatterer and the same material parameters as those in Table~\ref{comp4}. For the configuration presented in  Figure~\ref{fig:pade}, we have found in practice that the value $p=16$ leads to optimal iterative behavior of the DDM, but this behavior is sensitive to the values of $p$ in the high-frequency regime. Albeit smaller values of the Pad\'e parameter $p$ require less expensive evaluations of the transmission operators $Z_j^{Pade,p},j=0,1$, they lead to larger numbers of DDM iterations in the high-frequency regime. 
\begin{figure}
\centering
\includegraphics[height=60mm]{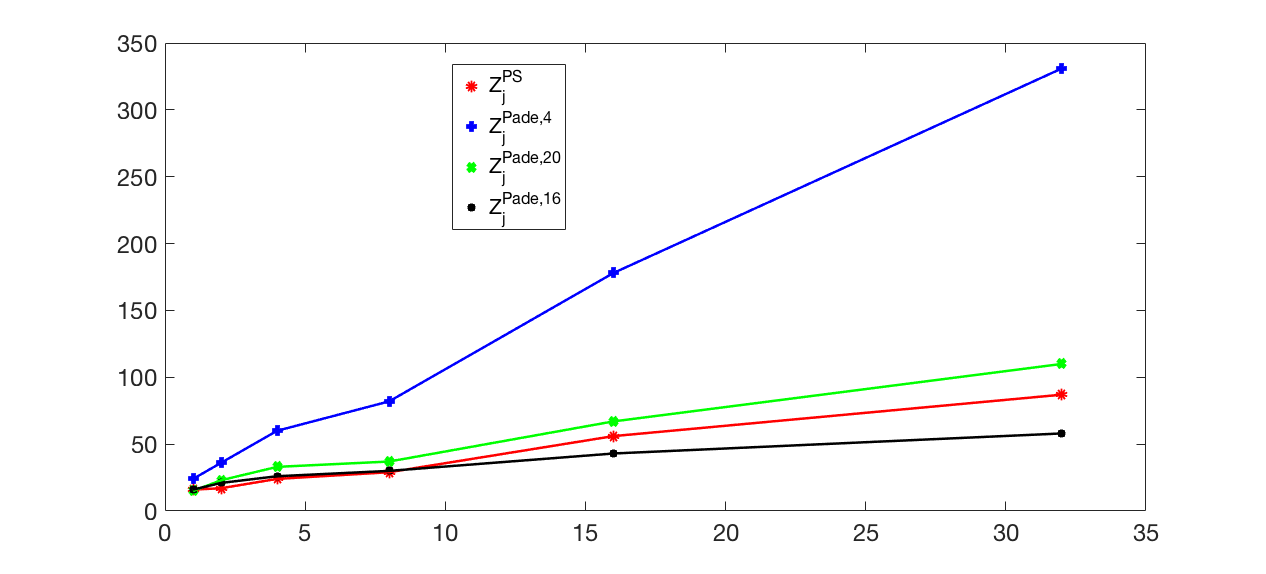}
\caption{The numbers of iterations required by the DDM solvers with transmission operators $Z_j^{PS},j=0,1$ as well as Pad\'e approximations $Z_j^{Pade,p}, j=0,1$ for various values of $p$, square scatterer and the same material parameters as those in Table~\ref{comp4}.}
\label{fig:pade}
\end{figure}

As it can be seen from the results in Tables~\ref{comp4} and~\ref{comp5}, the DDM solvers based on optimized transmission operators $Z_j$ and $Z_j^{PS}$ exhibit superior iterative Krylov subspace performance. Nevertheless, DDM formulations rely on discretization of RtR operators $\mathcal{S}^j$, which, in turn, require matrix inversions. As the frequency increases, the size of the matrices that need be inverted grows commensurably; furthermore, for three dimensional applications, the numbers of unknowns quoted in Tables~\ref{comp4} and~\ref{comp5} ought to be squared for the same acoustical/electrical size of domains. Clearly, a straightforward use of direct linear algebra solvers for computations of RtR operators is not possible in the high frequency regime. Therefore, we turn our attention in Tables~\ref{comp6} and~\ref{comp7} to the numbers of iterations required for computation of $\mathcal{S}^j$ corresponding to the transmission operators $Z_j,j=0,1$ based on the three formulations discussed in this text. Specifically, we used (1) interior/exterior formulations that require inversion of the operators $\mathcal{A}_j,j=0,1$ defined in equation~\eqref{eq:A_comp}; (2) interior/exterior formulations that require inversion of the operators $\mathcal{B}_j,j=0,1$ defined in equation~\eqref{eq:int_D}; and (3) interior formulations that require inversion of the operators $\mathcal{C}_1$ defined in equation~\eqref{eq:Steinbach} and exterior formulations that require inversion of the operators $\mathcal{C}_0$ defined in equation~\eqref{eq:Steinbach0}. Although there is no theory in place for the well-posedness of boundary integral equations that involve inversion of the operators $\mathcal{B}_0$ defined in equation~\eqref{eq:int_D}, our numerical experiments suggest that it is possible to invert discretizations of those operators. As it can be seen from the results presented in Tables~\ref{comp6} and~\ref{comp7}, while the numbers of iterations required to solve exterior impedance problems do not increase significantly with frequency provided that carefully defined formulations $\mathcal{A}_0$~\eqref{eq:A_comp} are used, this is no longer the case for interior impedance problems, regardless of formulation used. Similar scenarios occur for the other choices of transmission operators discussed in this text. As it can be seen from the results in  Tables~\ref{comp6} and~\ref{comp7}, the numbers of iterations required for the computation of the interior RtR map $\mathcal{S}^1$ cannot be controlled as the frequency increases, regardless of the use of any of the three BIE formulations considered in this text. We submit that this is related to the fact that easily computable approximations of DtN maps for interior domains (even when properly defined) are simply not available for high-frequencies.   

\begin{table}
  \begin{center}
    \resizebox{!}{1.6cm}
{   
\begin{tabular}{|c|c|c|c|c|c|c|c|}
\hline
$k_0$ & \multicolumn{3}{c|} {$\Omega_0$} & $k_1$ & \multicolumn{3}{c|} {$\Omega_1$}\\
\cline{2-4}\cline{6-8}
& $\mathcal{A}_0$~\eqref{eq:A_comp} & $\mathcal{B}_0$~\eqref{eq:int_D} & $\mathcal{C}_0$~\eqref{eq:Steinbach0} & & $\mathcal{A}_1$~\eqref{eq:A_comp} & $\mathcal{B}_1$~\eqref{eq:int_D} & $\mathcal{C}_1$~\eqref{eq:Steinbach} \\
\hline
1 & 13 & 16 & 37 & 4 & 18 & 21 & 49\\
2 & 17 & 21 & 49 & 8 & 26 & 29 & 70\\
4 & 24 & 36 & 84 & 16 & 51 & 56 & 131\\
8 & 31 & 49 & 104 & 32 & 83 & 79 & 217\\
16 & 35 & 75 & 143 & 64 & 170 & 142 & 431\\
32 & 42 & 125 & 228 & 128 & 263 & 214 & 793\\
\hline
\end{tabular}
}
\caption{ Numbers of iterations required for the calculation of the RtR operators $\mathcal{S}^j,j=0,1$ corresponding to the transmission operators $Z_j,j=0,1$ in the case of the square scatterer $\Omega_1$ using various boundary integral equation formulations discussed in this text .\label{comp6}}
\end{center}
 \end{table}

\begin{table}
  \begin{center}
    \resizebox{!}{1.6cm}
{   
\begin{tabular}{|c|c|c|c|c|c|c|c|}
\hline
$k_0$ & \multicolumn{3}{c|} {$\Omega_0$} & $k_1$ & \multicolumn{3}{c|} {$\Omega_1$}\\
\cline{2-4}\cline{6-8}
& $\mathcal{A}_0$~\eqref{eq:A_comp} & $\mathcal{B}_0$~\eqref{eq:int_D} & $\mathcal{C}_0$~\eqref{eq:Steinbach0} & & $\mathcal{A}_1$~\eqref{eq:A_comp} & $\mathcal{B}_1$~\eqref{eq:int_D} & $\mathcal{C}_1$~\eqref{eq:Steinbach} \\
\hline
1 & 17 & 22 & 44 & 4 & 24 & 26 & 67\\
2 & 22 & 27 & 58 & 8 & 38 & 42 & 92\\
4 & 31 & 39 & 80 & 16 & 66 & 65 & 160\\
8 & 34 & 63 & 131 & 32 & 106 & 94 & 247\\
16 & 38 & 104 & 188 & 64 & 218 & 195 & 473\\
32 & 45 & 168 & 309 & 128 & 405 & 333 & 890\\
\hline
\end{tabular}
}
\caption{Numbers of iterations required for the calculation of the RtR operators $\mathcal{S}^j,j=0,1$ corresponding to the transmission operators $Z_j,j=0,1$ in the case of the L-shaped scatterer $\Omega_1$ using various boundary integral equation formulations discussed in this text.\label{comp7}}
\end{center}
 \end{table}

Given the large computational costs required to compute the RtR operators $\mathcal{S}^1$ at high frequencies, it is preferrable that the interior domain $\Omega_1$ is decomposed in smaller non-overlapping subdomains giving rise to DDM formulations~\eqref{DDMnew}, in which case direct solvers such as LU can be used for the calculation of all the RtR maps required.  However, as shown in Figure~\ref{fig:iter_circle_subdomains}, the numbers of iterations grow considerally with the number of subdomains, albeit the computation of RtR maps becomes much more efficient since the electric size of interior subdomains has been decreased. This increase of number of iterations as the number of subdomains increases and the adjacency graph becomes more complex can be attributed to the global communication flow between subdomains, regardless of choice of transmission operators. This increase is more dramatic for the transmission operators $Z_j$ and $Z_j^{PS}$, and less so for the transmission operators $Z_j^a$. Still, the numbers of iterations required by the DDM with interior domain subdivisions~\eqref{DDMnew} and transmission operators $Z_j$ reported in Figure~\ref{fig:iter_circle_subdomains} are smaller than those corresponding to transmission operators $Z_j^{PS}$ and $Z_j^a$. Thus, even though the exchange of information between adjacent subdomains can be optimized, the number of DDM iterations does not scale with the number of subdomains, and preconditioners are needed to stabilize this phenomenon. The design of effective preconditioners for the DDM~\eqref{DDMnew} for Helmholtz equation with large numbers of subdomains that control the global interdomain communication is an active area of research. The most promising directions are (a) the use of coarse grid preconditioners~\cite{stolk2013rapidly,conen2014coarse} and (b) the use of  sweeping preconditioners~\cite{vion2014double}. The incorporation of these preconditioning strategies in the case of DDM~\eqref{DDMnew} is subject of ongoing investigation.

\begin{figure}
\centering
\includegraphics[height=60mm]{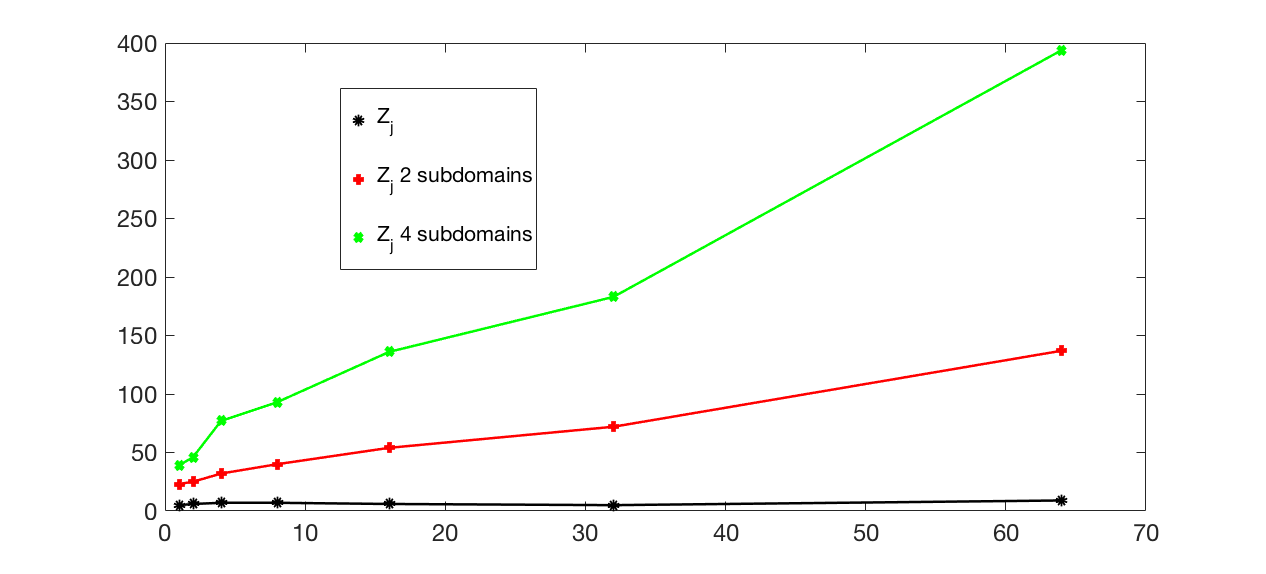}
\caption{The numbers of iterations required by the DDM solvers with transmission operators $Z_j,j=0,1$ in the case when the interior domain $\Omega_1$ is a circle of radius one that is divided into (a) two interior half-circle subdomains $\Omega_1=\Omega_{11}\cup\Omega_{12}$, and (b) four quater-circle interior subdomains $\Omega_1=\cup_{j=1}^4 \Omega_{1j}$. We used $\varepsilon_0=1$, $\varepsilon_1=16$, $\alpha_j=1,j=0,1$, and $\omega=1,2,4,8,16,32,64$. In the case when the L-shaped subdomain of size 4 is divided into three subdomains as depicted in Figure~\ref{fig:subdiv1}, the number of iterations of the DDM algorithm~\eqref{DDMnew} are (i) 58, 72, 106, 171, 306, and respectively 447 for transmission operators $Z_j$, (ii) 65, 83, 118, 179, 330, and respectively 509 for transmission operators $Z_j^{PS}$, and (iii) 103, 133, 179, 286, 533, and respectively 740 for transmission operators $Z_j^a$, for frequencies $\omega=1,2,4,8,16,32$ and material properties described in Table~\ref{comp5}. }
\label{fig:iter_circle_subdomains}
\end{figure}

We conclude with numerical experiments concerning transmission problems with partially coated boundaries. Specifically, we present in Table~\ref{comp10} numbers of iterations required by the CFIESK formulation~\eqref{eq:sk_system} and DDM with various transmission operators considered in this text. The domain $\Omega_1$ is a circle of radius one whose lower semicircle is coated. In this case, the reductions in numbers of iterations that can be garnered from use of DDM over the use of CFIESK is more pronounced. Finally, we plot in Figure~\ref{fig:near_fields} fields scattered by a penetrable scatterer whose boundary is partially coated under plane wave incidence of various directions and frequencies.

\begin{table}
   \begin{center}
     \resizebox{!}{1.2cm}
{   
\begin{tabular}{|c|c|c|c|c|c|c|c|c|}
\hline
$\omega$ & \multicolumn{2}{c|} {CFIESK} & \multicolumn{2}{c|} {DDM $Z_j,j=0,1$} &  \multicolumn{2}{c|} {DDM $Z_j^{PS},j=0,1$} &\multicolumn{2}{c|} {DDM $Z_j^a,j=0,1$}\\
\cline{2-9}
& It & $\varepsilon_\infty$ &It & $\varepsilon_\infty$ & It & $\varepsilon_\infty$ & It & $\varepsilon_\infty$ \\
\hline
1 & 85 & 4.7 $\times$ $10^{-3}$ & 13 & 6.2 $\times$ $10^{-3}$ & 12 & 6.2 $\times$ $10^{-3}$ & 21 & 6.3 $\times$ $10^{-3}$\\
2 & 165 & 5.2 $\times$ $10^{-3}$ & 17 &  6.8 $\times$ $10^{-3}$ & 15 & 6.9 $\times$ $10^{-3}$ & 34 & 6.8 $\times$ $10^{-3}$\\
4 & 315 & 5.8 $\times$ $10^{-3}$ &  17 & 7.4 $\times$ $10^{-3}$ & 18 & 7.4 $\times$ $10^{-3}$ & 37 & 7.3 $\times$ $10^{-3}$\\
8 & 617 & 6.1 $\times$ $10^{-3}$ & 19 & 7.6 $\times$ $10^{-3}$ & 23 & 7.6 $\times$ $10^{-3}$ & 52 & 7.5 $\times$ $10^{-3}$\\
16 & 1225 & 6.8 $\times$ $10^{-3}$ & 21 & 7.8 $\times$ $10^{-3}$ & 29 & 7.8 $\times$ $10^{-3}$ & 118 & 7.9 $\times$ $10^{-3}$\\
32 & 2271 & 7.2 $\times$ $10^{-3}$ & 23  & 8.5 $\times$ $10^{-3}$ & 44 & 8.5 $\times$ $10^{-3}$ & 265 & 8.4 $\times$ $10^{-3}$\\
\hline
\end{tabular}
}
\caption{Far-field errors computed using various formulations considered in this text in the case of scattering from a circle of radius one with $\varepsilon_0=1$ and $\varepsilon_1=16$ with $\alpha_j=1, j=0,1$, and the lower semi-circle is PEC. The DDM discretization used $64, 128, 256, 512, 1024$ and respectively $2048$ unknonws; CFIESK formulations used twice as many unknowns. In the case when the domain $\Omega_1$ is further subdivided into two subdomains $\Omega_{11}$ and $\Omega_{12}$ the numbers of DDM iterations are (i) $24, 33, 39, 56, 95, 173$ for transmission operators $Z_j$, (ii) $22, 31, 43, 69, 135, 256$ for transmission operators $Z_j^{PS}$, and (iii) $34, 63,73,125,251,529$ for transmission operators $Z_j^a$ for the same frequencies and material parameters.\label{comp10}}
\end{center}
\end{table}

\begin{figure}
\centering
\includegraphics[height=60mm]{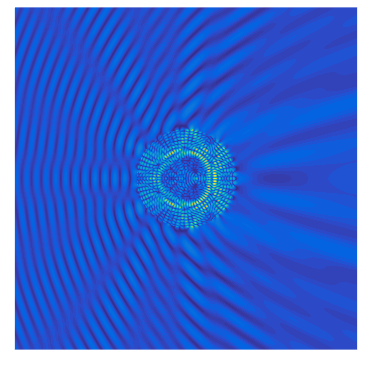}\includegraphics[height=60mm]{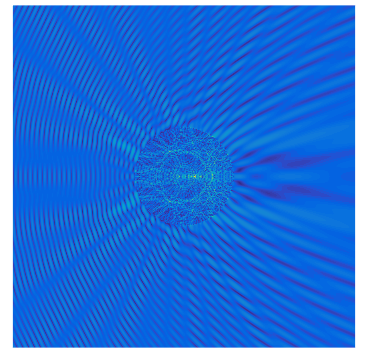}\\
\includegraphics[height=60mm]{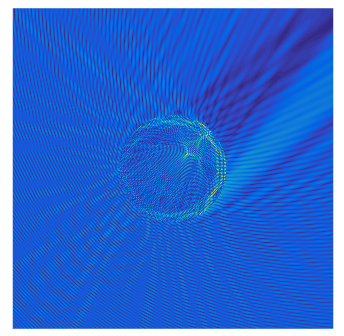}\includegraphics[height=60mm]{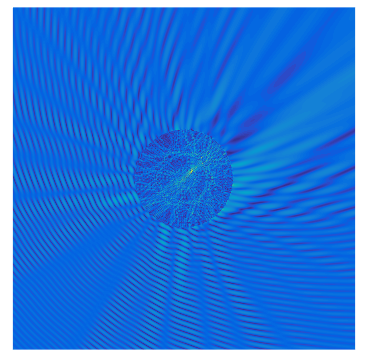}
\caption{Fields scattered by a circular structure filled with a material with $\varepsilon_1=16$ and PEC lower semicircle in the case of $\omega=k_0=16$ (top left), $\omega=k_0=32$ (top and bottom right), and $\omega=k_0=64$ (bottom left) and various plane wave incident fields.}
\label{fig:near_fields}
\end{figure}

\section{Conclusions}\label{conclu}

We presented analysis and numerical experiments concerning DDM based on quasi-optimal transmission operators for the solution of Helmholtz transmission problems in two dimensions. The quasi-optimal transmission operators that we used are readily computable approximations of DtN operators. Under certain assumptions on the regularity of the of the (closed) boundary of material discontinuity we established the well posedness of the DDM with the transmission operators considered. We provided ample numerical evidence that the incorporation of quasi-optimal transmission operators within DDM gives rise to small numbers of Krylov subspace iterations for convergence that depend very mildly on the frequency or contrast. However, the numers of iterations do not scale with the number of subdomains involved in the DDM. Extensions to three-dimensional configurations are currently underway. 

\section*{Acknowledgments}
Yassine Boubendir gratefully acknowledges support from NSF through contracts DMS-1720014. Catalin Turc gratefully acknowledges support from NSF through contracts DMS-1614270. Carlos Jerez-Hanckes thanks partial support from Conicyt Anillo ACT1417 and Fondecyt Regular 1171491.

\bibliography{biblioTJ}

\end{document}